\documentclass[a4paper,pdftex,reqno,10pt]{amsart}

\allowdisplaybreaks[1]

\usepackage{geometry}
\usepackage{pdflscape}
\usepackage{graphicx}
\usepackage{enumerate}
\usepackage{courier}
\usepackage{times}
\usepackage{amsmath,amssymb}

\usepackage{hyperref}

\theoremstyle{plain}
\newtheorem{Theorem}{Theorem}[section]
\newtheorem{Proposition}[Theorem]{Proposition}

\newtheorem{Corollary}[Theorem]{Corollary}
\newtheorem{Lemma}[Theorem]{Lemma}

\newenvironment{Proof}
{\begin{trivlist}\item[]{{\sc Proof.}}}{\hfill{$\square$}\noindent\end{trivlist}}

\theoremstyle{definition}
\newtheorem{Definition}[Theorem]{Definition}

\newtheorem{Example}[Theorem]{Example}

\theoremstyle{remark}
\newtheorem{Remark}[Theorem]{Remark}

\newcommand{\PG}{\operatorname{PG}}
\newcommand{\F}{\mathbb{F}}
\newcommand{\N}{\mathbb{N}}
\newcommand{\cS}{\mathcal{S}}
\newcommand{\cM}{\mathcal{M}}
\newcommand{\cP}{\mathcal{P}}
\newcommand{\cX}{\mathcal{X}} 
\newcommand{\cV}{\mathcal{V}}

\begin{document}

%%%%%%%%%%%%%%%%%%%%%%%%%%%%%%%%%%%%%%%%%%%%%%%%%%%%%%%%%%%%%%%%%%%%%%%%%%%%%%
%% Title:
%%%%%%%%%%%%%%%%%%%%%%%%%%%%%%%%%%%%%%%%%%%%%%%%%%%%%%%%%%%%%%%%%%%%%%%%%%%%%%

\title[O\MakeLowercase{ptimal} \MakeLowercase{additive} \MakeLowercase{quaternary} \MakeLowercase{codes} \MakeLowercase{of} \MakeLowercase{dimension} $3.5$]{O\MakeLowercase{ptimal} \MakeLowercase{additive} \MakeLowercase{quaternary} \MakeLowercase{codes} \MakeLowercase{of} \MakeLowercase{dimension} $\mathbf{3.5}$ \MakeLowercase{and} $\mathbf{4}$}

\author{\small S\MakeLowercase{ascha} K\MakeLowercase{urz}\\ \footnotesize U\MakeLowercase{niversity} \MakeLowercase{of} B\MakeLowercase{ayreuth}, G\MakeLowercase{ermany}\\ \MakeLowercase{sascha.kurz@uni-bayreuth.de}}

\date{}

\begin{abstract}
  After the optimal parameters of additive quaternary codes of dimension $k\le 3$ have been determined in \cite{bierbrauer2021optimal}, 
  there was some activity to settle the next case of dimension $k=3.5$ \cite{guan2023some,10693309}. Here we complete dimensions $k=3.5$ 
  and %% give partial results for dimension 
  $k=4$. We also solve the problem of the optimal parameters of additive quaternary codes of arbitrary 
  dimension when assuming a sufficiently large minimum distance. 
  
  \medskip
  
  
  \smallskip
  
\end{abstract}

\maketitle

\section{Introduction}
\label{sec_intro}

A quaternary block code $C$ of length $n$ is a subset of $\F_4^n$. If $C$ is closed under componentwise addition then $C$ is called additive. If $C$ is additive and 
closed under $\F_4$ scalar multiplication then $C$ is called linear. The parameter $k$ such that the number of codewords $|C|$ equals $4^k$ is called the dimension of $C$. 
Clearly, $k$ is an integer if $C$ is linear and a half-integer if $C$ is additive. For each integer $s$ let $n_k(s)$ denote the maximal length $n$ such that 
an additive quaternary code of length $n$, dimension $k$, and minimum Hamming distance $n-s$ exists. For $k\le 3$ the function $n_k(s)$ was completely determined in 
\cite{bierbrauer2021optimal}. In the sequence of papers \cite{guan2023some,10693309} the determination of $n_{3.5}(s)$ was narrowed down to $s\in \{6,7,12\}$.\footnote{The example for $s=13$ 
refers to \cite{kurz2024computer}.} Geometrically, $n_k(s)$ is the maximum number of lines in the projective space $\PG(2k-1,2)$ such that each hyperplane contains at most $s$ lines, 
see e.g.\ \cite{bierbrauer2021optimal}. The aim of this paper is to completely determine $n_{3.5}(s)$, $n_{4}(s)$, and $n_k(s)$ for all sufficiently large $s$. 

The remaining part of the paper is structured as follows. In Section~\ref{sec_preliminaries} we introduce the necessary preliminaries. The
problem of the optimal parameters of additive quaternary codes of arbitrary dimension, assuming a sufficiently large minimum distance, is solved in 
Section~\ref{sec_solomon_stiffler}, see Corollary~\ref{cor_attained_asymptotically}. The determination of $n_{3.5}(s)$ and %%the partial determination of 
$n_{4}(s)$ is 
obtained in Section~\ref{sec_exact}. It turns out that there exist infinite series of additive codes whose parameters outperform those of linear codes.  

\section{Preliminaries}
\label{sec_preliminaries}
The set of all subspaces of $\F_2^r$, ordered by the incidence relation
$\subseteq$, is called \emph{$(r-1)$-dimensional projective geometry over
$\F_2$} and denoted by $\PG(r-1,2)$. Employing this algebraic notion
of dimension instead of the geometric one, we will use the term $i$-space
to denote an $i$-dimensional subspace of $\F_2^r$. To highlight the
important geometric interpretation of subspaces we will call $1$-, $2$-,
and $(r-1)$-spaces points, lines, and hyperplanes, respectively. Every 
$i$-space in $\PG(r-1,2)$, where $r\ge i$, contains exactly $2^i-1$ 
points. For two subspaces $S$ and $S'$ we write $S\subseteq S'$ if $S$ is contained 
in $S'$. Moreover, we say that $S$ and $S'$ are \emph{incident} iff $S\subseteq S'$ 
or $S\supseteq S'$.
\begin{Definition}
  An \emph{$(n,r,s)$ system} is a multiset $\cS$ of $n$ lines in $\PG(r-1,2)$ 
  such that each hyperplane contains at most $s$ elements from $\cS$ and some 
  hyperplane contains exactly $s$ elements of $\cS$. We say that $\cS$ is 
  \emph{spanning} iff $s<n$. %% An \emph{$(n,r,s,\mu)$} system $\cS$ is an 
  %% $(n,r,s)$ system such that each point is contained in at most $\mu$ elements 
  %% from $\cS$ and some point is contained in exactly $\mu$ elements from $\cS$. 
\end{Definition}
By $n_k(s)$ we denote the maximum $n$ such that a spanning $(n,2k,s)$ system exists, which is the same as the 
maximal length $n$ of an additive quaternary code with dimension $k$ and minimum Hamming distance $n-s$, 
see e.g.\ \cite{bierbrauer2021optimal}. So, we will always assume $2(s+1)\ge k$ when considering $n_k(s)$.   
\begin{Definition}
  For an $(n,r,s)$ system $\cS$ let $\cP(\cS)$ denote the multiset of points that we obtain
  by replacing each element of $\cS$ by its contained three points.
\end{Definition}
We also call a multiset of points spanning iff no hyperplane contains all points. If $C$ is a binary linear code with length $n$ 
and minimum Hamming distance $d$, then we say that $C$ is an $[n,k,d]_2$ code, where $2^k=|C|$. Given a generator matrix $G$ 
for $C$ we can construct a multiset of points from $C$ by considering the span of each column of $G$ as a point in $\PG(k-1,2)$. 
And indeed, it is well known that a spanning multiset of points $\cP$ in $\PG(k-1,2)$, such that at most $s$ elements are contained 
in a hyperplane and some hyperplane contains exactly $s$ elements, is in one-to-one correspondence to a linear $[n,k,n-s]_2$ code $C$, 
see e.g.\ \cite[{\S}1.1.2]{tsfasman1991agcodes} or \cite{dodunekov1998codes}. Let us write $\cP=\cX(C)$ and $C=\cX^{-1}(\cP)$ for this 
correspondence. The elements of a code are called \emph{codewords}. The \emph{weight} of a codeword $c\in C$ of a linear code is the
number of non-zero entries in $c$. So, the minimum occurring non-zero weight of a linear code coincides with its minimum distance.
We call a linear code \emph{$\Delta$-divisible} if the weights of all codewords are divisible by $\Delta$. If $\Delta$ equals $2$ or $4$ then 
we also speak of even and doubly-even codes, respectively. 
\begin{Lemma}(Cf.~\cite[Lemma 1]{bierbrauer2021optimal})
  \label{lemma_binary_code}
  Let $\cS$ be a spanning $(n,r,s)$ system. Then, $C:=\cX^{-1}(\cP(\cS))$ is a $2$-divisible $[3n,r,2(n-s)]_2$ code with maximum weight at most $2n$. 
\end{Lemma}
\begin{proof}
  Since each line consists of $2^2-1=3$ points the cardinality of $\cP(\cS)$ equals $3n$, so that $C$ has length $3n$. Since $\cS$ is spanning also $\cP(\cS)$ is 
  spanning and $C$ has dimension $r$. Given an arbitrary hyperplane $H$ in $\PG(r-1,2)$ and an arbitrary line $L$ we have that either $L$ is completely contained in $H$ 
  or intersects the hyperplane in exactly a point. Since each hyperplane $H$ contains $0\le i\le s$ out of the $n$ lines in $\cS$ we have that $H$ contains 
  $3i+(n-i)=n+2i$ points from $\cP(\cS)$, so that the codeword $c\in C$ that corresponds to $H$ has weight $(3n)-(n+2i)=2(n-i)$.  
\end{proof}
So, bounds on the parameters of a binary linear code yield upper bounds for $n_k(s)$. E.g.\ the so-called
\emph{Griesmer bound} \cite{griesmer1960bound}
\begin{equation}
  \label{eq_griesmer_bound}
  n\ge \sum_{i=0}^{k-1} \left\lceil\frac{d}{2^i}\right\rceil=:g(k,d)
\end{equation}   
relates the parameters of an $[n,k,d]_2$ code. If $n=g(k,d)$ we speak of Griesmer codes. Interestingly enough, Griesmer codes always exist if the minimum distance $d$ is sufficiently large 
and a nice geometric construction was given by Solomon and Stiffler \cite{solomon1965algebraically}.
\begin{Lemma} (Cf.~\cite[Theorem 12]{ball2024additive}, \cite[Lemma 1]{bierbrauer2021optimal}, and \cite[Lemma 3]{guan2023some})  
  \label{lemma_indirect_upper_bound}
  Let $r>2$ and $\cS$ be an $(n,r,s)$ system. Then, we have $g\!\left(r,2(n-s)\right)\le 3n$.
\end{Lemma}
\begin{proof}
  Combine Lemma~\ref{lemma_binary_code} with Inequality~(\ref{eq_griesmer_bound}).
\end{proof}
In other words, we have 
\begin{equation}
  \label{ie_griesmer_uppper_bound}
  n\ge \left\lceil \frac{g(k,2d)}{3}\right\rceil
  =\left\lceil \frac{\sum_{i=0}^{k-1}\left\lceil\frac{2d}{2^i}\right\rceil}{3}\right\rceil
  =\left\lceil \frac{2d+g(k-1,d)}{3}\right\rceil   
  =d+\left\lceil \frac{g(k-1,d)-d}{3}\right\rceil
  \!,
\end{equation}
where $d=n-s$ and $k=r$.
\begin{Definition}
  \label{def_griesmer_uppper_bound}
  The \emph{Griesmer upper bound} for $n_k(s)$ is the largest integer
  $n$ such that $g\!\left(2k,2(n-s)\right)\le 3n$. The \emph{weak 
  coding upper bound} for $n_k(s)$ is the largest integer $n$ such that 
  a $[3n,2k,2(n-s)]_2$ code $C$ exists. The \emph{(strong) coding 
  upper bound} for $n_k(s)$ is the largest integer $n$ such that a 
  $2$-divisible $[3n,2k,2(n-s)]_2$ code $C$ with maximum weight at most 
  $2n$ exists.  
\end{Definition}
We remark that the minimal possible length of an $[n,k,d]_2$ code is known for all $d\in\N$ when $k\le 8$ \cite{bouyukhev2000smallest}, so 
that the weak coding upper bound can be easily evaluated for $n_4(s)$. This is different for the strong coding upper bound for $n_4(s)$, see 
Subsection~\ref{sec_nonexistence}. 

\begin{table}[htp]
  \begin{center}
     \begin{tabular}{rrrr}
       \hline
       $s$ & Griesmer upper bound & weak coding upper bound & $n_4(s)$ \\
       \hline
        3 &   9 &  7 &  5 \\
        4 &  12 &    & 10 \\
        5 &  17 &    & 17 \\
        6 &  22 & 18 & 18 \\
        7 &  25 & 23 & 23 \\
        8 &  30 & 28 & 28 \\
        9 &  33 &    & 33 \\
       10 &  38 & 36 & 36 \\
       11 &  43 & 40 & 40 \\
       12 &  44 &    & 44 \\
       13 &  49 &    & 49 \\
       14 &  54 &    & 54 \\
       15 &  59 & 57 & 55 \\
       %%24 &  94 &    & 93--94 \\
       28 & 110 &    & 108 \\
       29 & 115 &    & 113 \\
       \hline
     \end{tabular}
     \caption{The Griesmer and the weak coding upper bound for $n_4(s)$.}
     \label{table_griesmer_upper_bound}
  \end{center}
\end{table}

\begin{Example}
  %%As an example we state that
  The Griesmer upper bound for $n_4(8)$ is $30$ and the weak coding upper bound is $28$.
  I.e., the Griesmer bound implies that no $[93,8,46]_2$ code exists but cannot rule
  out the existence of a $[90,8,44]_2$ code, so that $n_4(8)\le 30$ is the
  sharpest upper bound we can deduce from the Griesmer bound (for linear codes). However,
  since the existence of a $[84,8,40]_2$ code and the non-existence of a $[87,8,42]_2$
  code is known, we obtain $n_4(8)\le 28$. In Table~\ref{table_griesmer_upper_bound} we 
  list the Griesmer and the weak coding upper bound for $n_4(s)$ for $3\le s\le 15$ and 
  all cases were either the weak coding upper bound is strictly less than the Griesmer 
  upper bound or $n_4(s)$ is strictly less than the weak coding upper bound. For $s\in\{3,4\}$ we refer to
  \cite{blokhuis2004small}. Note that the cases $s\in\{1,2\}$ cannot occur for a spanning 
  $(n,8,s)$ system. We do not display the weak coding upper bound when it coincides 
  with the Griesmer upper bound.
\end{Example}

In order to partially evaluate the strong coding upper bound we present a few tools from coding theory. Let 
$C$ be a $[n,k,d]_2$ code with generator matrix $G$ and $c\in C$ be a codeword of weight $w$. The 
\emph{residual code of $C$ with respect to $c$}, denoted by $\operatorname{Res}(C;c)$, is the 
code generated by the restriction of $G$ to the columns where $c$ has a zero entry. If only the 
weight $w$ of $c$ is relevant we will denote it by $\operatorname{Res}(C;w)$. The following 
statement on the residual code is well-known:
\begin{Lemma}
  \label{lemma_residual_code}
  Let $C$ be an $[n,k,d]_2$ code and let $d>\tfrac{w}{2}$. Then $\operatorname{Res}(C;w)$ is an 
  $[n-w,k-1,\ge d-\left\lfloor w/2\right\rfloor]_2$ code.
\end{Lemma}   

\begin{Lemma}
  (\cite[Theorem 1]{ward1998divisibility}) 
  Let $C$ be an $[n,k,d]_2$ code with $n=g(k,d)$. If $2^e$ divides $d$, then $C$ is $2^e$-divisible.
  \label{lemma_griesmer_divisibility} 
\end{Lemma} 

\begin{Proposition} (\cite{macwilliams1977theory}, MacWilliams Identities)
  \label{prop_MacWilliams_identitites}
   Let $C$ be an $[n,k,d]_2$ code and $C^\perp$ be the dual code of $C$. Let $A_i(C)$ and $B_i(C)$ be the number of codewords of weight $i$ in $C$ and
   $C^\perp$, respectively. With this, we have
   \begin{equation}
     \label{eq_MacWilliams}
     \sum_{j=0}^n K_i(j)A_j(C)=2^kB_i(C),\quad 0\le i\le n
   \end{equation}
   where
   $$
     K_i(j)=\sum_{s=0}^n (-1)^s {{n-j}\choose{i-s}}{{j}\choose{s}},\quad 0\le i\le n
   $$
   are the binary Krawtchouk polynomials. We will simplify the notation to $A_i$ and $B_i$ whenever $C$ is clear from the context.
 \end{Proposition}
 Whenever we speak of the first $l$ MacWilliams identities, we mean Equation~(\ref{eq_MacWilliams}) for $0\le i\le l-1$.
 Adding the non-negativity constraints $A_i,B_i\ge 0$ we obtain a linear program where we can maximize or minimize certain quantities, which is called the
 linear programming method for linear codes. Adding additional equations or inequalities strengthens the formulation.
 For an $[n,k,d]_2$ code of full length, i.e.\ $B_1=0$, we can rewrite the first four MacWilliams identities to
 \begin{eqnarray}
   \sum_{i>0} A_i   &=& 2^k-1,\label{eq_mw1}\\
   \sum_{i} iA_i    &=& 2^{k-1}n, \label{eq_mw2}\\
   \sum_{i} i^2 A_i &=& 2^{k-1}\cdot\left(B_2+n(n+1)/2\right),\label{eq_mw3}\\
   \sum_{i} i^3 A_i &=& 2^{k-2}\cdot\left(3(B_2n-B_3)+n^2(n+3)/2\right).\label{eq_mw4}
 \end{eqnarray}
 The weight enumerator can be generalized to the split weight enumerator based on a partition of the coordinates \cite{simonis1995macwilliams}. For a linear
 programming method based on the split enumerator we refer e.g.\ to \cite{jaffe1997brief}.

 \begin{Proposition}(\cite[Proposition 5]{dodunekov1999some})
   \label{prop_div_one_more}
   Let $C$ be an $[n,k,d]_2$ code with all weights divisible by $\Delta:=2^a$ and let $\left(A_i\right)_{i=0,1,\dots,n}$ be the weight distribution of $C$. Put
   \begin{eqnarray*}
     \alpha&:=&\min\{k-a-1,a+1\},\\
     \beta&:=&\lfloor(k-a+1)/2\rfloor,\text{ and}\\
     \delta&:=&\min\{2\Delta i\,\mid\,A_{2\Delta i }\neq 0\wedge i>0\}.
   \end{eqnarray*}
   Then the integer
   $$
     T:=\sum_{i=0}^{\lfloor n/(2\Delta)\rfloor} A_{2\Delta i}
   $$
   satisfies the following conditions.
   \begin{enumerate}
     \item \label{div_one_more_case1}
           $T$ is divisible by $2^{\lfloor(k-1)/(a+1)\rfloor}$.
     \item \label{div_one_more_case2}
           If $T<2^{k-a}$, then
           $$
             T=2^{k-a}-2^{k-a-t}
           $$
           for some integer $t$ satisfying $1\le t\le \max\{\alpha,\beta\}$. Moreover, if $t>\beta$, then $C$ has an $[n,k-a-2,\delta]_2$ subcode and if $t\le \beta$, it has an
           $[n,k-a-t,\delta]_2$ subcode.
     \item \label{div_one_more_case3}
           If $T>2^k-2^{k-a}$, then
           $$
             T=2^k-2^{k-a}+2^{k-a-t}
           $$
           for some integer $t$ satisfying $0\le t\le \max\{\alpha,\beta\}$. Moreover, if $a=1$, then $C$ has an $[n,k-t,\delta]_2$ subcode. If $a>1$, then $C$ has an
           $[n,k-1,\delta]_2$ subcode unless $t=a+1\le k-a-1$, in which case it has an $[n,k-2,\delta]_2$ subcode.
   \end{enumerate}
 \end{Proposition}

For the constructive lower bound we have:
\begin{Lemma}
  \label{lemma_union}
  For $k>1$ we have $n_k(s_1+s_2)\ge n_k(s_1)+n_k(s_2)$ and $n_k(s+1)\ge n_k(s)+1$.
\end{Lemma}
\begin{proof}
  Let $\cS_i$ be spanning $(n_i,2k,s_i)$ systems for $i=1,2$ and $\cS$ a spanning  $(n,2k,s)$ system. With this, the multiset union of $\cS_1$ and $\cS_2$ is a spanning 
  $(n_1+n_2,2k,\le s_1+s_2)$ system. Adding an arbitrary line to $\cS$ gives a spanning $(n+1,2k,\le s+1)$ system.  
\end{proof}

\begin{Definition}
  A \emph{vector space partition} of $\PG(r-1,2)$ is a multiset $\cV$ of subspaces with dimension at most $(r-1)$ such that every point of $\PG(r-1,2)$ is contained in exactly
  one element of $\cV$.  We say that $\cV$ has type $1^{t_1} 2^{t_2}\dots (r-1)^{t_{r-1}}$ if exactly $t_i$ elements of $\cV$ have dimension $i$ for all $1\le i\le r-1$.
\end{Definition}

A set of matrices $M\subseteq \F_2^{m\times n}$ with $\operatorname{rk}(A-B)\ge \delta$
for all $A,B\in M$ with $A\neq B$ is called a \emph{rank metric code} with minimum rank
distance $\delta$. A Singleton-type upper bound gives $|M|\le 2^{\max\{m,n\}\cdot(\min\{m,n\}-\delta+1)}$. 
Rank metric codes attaining this bound are called \emph{MRD} codes.
They exist for all parameters with $\delta\le\min\{m,n\}$, even if one additionally
requires that $M$ is linearly closed, see e.g.\ \cite{sheekey201913} for a survey.

\begin{Lemma}
  \label{lemma_lifting}
  For $r>4$ there exists a vector space partition $\cV$ of $\PG(r-1,2)$ of type
  $2^{t_2} (r-2)^1$ where $t_2=2^{r-2}$.
\end{Lemma}
\begin{proof}
  Let $M\subseteq \F_2^{2\times (r-2)}$ be an MRD code with minimum rank distance $2$ and cardinality $2^{r-2}$. Prepending a $2\times 2$ unit matrix to the elements of $M$
  gives generator matrices of $2$-spaces in $\PG(r-1,2)$ that are pairwise disjoint and disjoint to an $(r-2)$-space $S$.
\end{proof}

\begin{Lemma}
  \label{lemma_vsp}
  For $r>a> 2$ with $r\equiv a\pmod 2$ there exists a vector space partition $\cV$ of $\PG(r-1,2)$ of type $2^{t_2} a^1$ where $t_2=2^{a}\cdot\tfrac{2^{r-a}-1}{3}$.
\end{Lemma}
\begin{proof}
  We prove by induction over $r$. Let $\cV$ be the vector space partition
  obtained from Lemma~\ref{lemma_lifting} and let $S\in\cV$ be the unique
  $(r-2)$-dimensional element. If $a=r-2$, which is indeed the case for all
  $r<6$, then $\cV$ is the desired vector space partition. Otherwise we
  identify $S$ with $\PG(r-3,2)$ and replace $S$ by a vector space partition
  of $\PG(r-3,2)$ of type $2^{t_2'} a^1$, which exists by induction.
\end{proof}

\begin{Lemma}
  \label{lemma_construction_x_preparation}
  For $r>a>2$ with $r\equiv a\pmod 2$ let $\cS$ be the set of $2$-dimensional
  elements of a vector space partition $\cV$ of $\PG(r-1,2)$ of type
  $2^{t_2}a^1$ and $A$ be the unique $a$-dimensional element in $\cV$. Then,
  $\cS$ is a $(t_2,r,s)$ system where
  $t_2=2^{a}\cdot\tfrac{2^{r-a}-1}{3}$ and
  $s=2^{a-2}\cdot \tfrac{2^{r-a}-1}{3}$. Moreover, each hyperplane
  that contains $A$ contains $s-2^{a-2}$ elements from $\cS$.
\end{Lemma}
\begin{proof}
  Let $H$ be an arbitrary hyperplane of $\PG(r-1,2)$. Note that every
  $i$-space intersects $H$ in either $2^i-1$ or $2^{i-1}-1$ points and that
  the elements of $\cS$ partition the points outside of $A$. Counting
  points yields that $H$ contains
  $$
    \frac{2^{r-1}-2^{a-1}-t_2}{2}=
    2^{a-2}\cdot \frac{2^{r-a}-1}{3}=s
  $$
  elements from $\cS$ if $A\subsetneq H$ and
  $$
    \frac{2^{r-1}-2^a-t_2}{2}=
    \frac{2^{r-1}-2^{a-1}-2^{a-1}-t_2}{2}=s-2^{a-2}
  $$
  elements from $\cS$ if $A\subseteq H$.
\end{proof}

%% \begin{Lemma}
%%   \label{lemma_construction_x}(Cf.~\cite[Lemma 6]{guan2023some})
%%   Let $\cS_1$ be an $\left(n_1,r,s_1\right)$ system, $A$ be an $a$-space such that each hyperplane that contains
%%   $A$ contains at most $s_0$ elements from $\cS_1$, and $\cS_2$ be an $\left(n_2,a,s_2\right)$
%%   system. Then, we have
%%   $n_{r/2}\!\left(\max\!\left\{s_1+s_2,s_0+n_2\right\}\right)\ge n_1+n_2$.
%% \end{Lemma}
%% \begin{proof}
%%   We identify $A$ with $\PG(a-1,q)$ and insert $\cS_2$ into the subspace $A$ of $\cS_1$.
%% \end{proof}
%% In analogy with linear codes, the authors of \cite{guan2023some,10693309} speak of the {\lq\lq}additive construction X{\rq\rq}. We give an application in
%% Lemma~\ref{lemma_construction_x_consequence}.

\section{A generalization of the Solomon--Stiffler construction}
\label{sec_solomon_stiffler}

In \cite{solomon1965algebraically} Solomon and Stiffler constructed $[n,k,d]_2$ codes with $n=g(k,d)$ for all parameters with sufficiently
large minimum distance $d$. Here we want to show the generalization that the Griesmer upper bound for $n_{k/2}(s)$ can always be attained 
if $s$ is sufficiently large. Using a specific parameterization of the minimum distance $d$, the Griesmer bound in 
Inequality~(\ref{eq_griesmer_bound}) can be written more explicitly:
\begin{Lemma}
  \label{lemma_parameters_griesmer_code}
  Let $k$ and $d$ be positive integers. Write $d$ as
  \begin{equation}
    \label{eq_griesmer_representation_min_dist}
    d=\sigma \cdot 2^{k-1}-\sum_{i=1}^{k-1}\varepsilon_i \cdot 2^{i-1},
  \end{equation}
  where $\sigma\in\N_0$ and $\varepsilon_i\in\{0,1\}$ for all $1\le i\le k-1$. Then, Inequality~(\ref{eq_griesmer_bound})
  is satisfied with equality iff
  \begin{equation}
    \label{eq_griesmer_representation_length}
    n=\sigma\cdot\left(2^k-1\right)-\sum_{i=1}^{k-1}\varepsilon_i\cdot\left(2^i-1\right),
  \end{equation}
  which is equivalent to
  \begin{equation}
    \label{eq_griesmer_representation_species}
    n-d=\sigma\cdot\left(2^{k-1}-1\right)-\sum_{i=1}^{k-1}\varepsilon_i\cdot\left(2^{i-1}-1\right).
  \end{equation}
\end{Lemma}
Given $k$ and $d$, Equation~(\ref{eq_griesmer_representation_min_dist})
  always determines $\sigma$ and the $\varepsilon_i$ uniquely. This is
  different for Equation~(\ref{eq_griesmer_representation_species}) given
  $k$ and $n-d=s$. Here it may happen that no solution with
  $0\le \varepsilon_i\le 1$ exists. By relaxing to
  $0\le \varepsilon_i\le 2$ we can ensure existence and uniqueness is
  enforced by additionally requiring $\varepsilon_j=0$ for all $j<i$
  where $\varepsilon_i=2$ for some $i$. The same is true for
  Equation~(\ref{eq_griesmer_representation_length}) given $k$ and $n$.
  For more details we refer to \cite[Chapter 2]{govaerts2003classifications}
  which also gives pointers to Hamada's work on minihypers.

\begin{Definition}
  \label{def_partitionable}
  Let $\sigma\in\N$, $\varepsilon_1,\dots,\varepsilon_{r-1}\in\mathbb{Z}$, and let
  $V$ denote the $r$-dimensional ambient space $\PG(r-1,2)$.
  We say that an $(n,r,s)$ system $\cS$ has \emph{type $\sigma[r]-\sum_{i=1}^{r-1}\varepsilon_i [i]$} if there 
  exist subspaces $S_1\subseteq \dots\subseteq S_{r-1}$ with $\dim\!\left(S_i\right)=i$ and
  \begin{equation}
    \sum_{S\in\cS} \chi_S=\sigma\cdot\chi_V-\sum_{i=1}^{r-1}\varepsilon_i\cdot \chi_{S_i},
  \end{equation}
  where $\chi_S$ denotes the characteristic function of a subspace $S$, i.e., $\chi_S(P)=1$ iff $P\subseteq S$ for every point $P$ and $\chi_S(P)=0$ otherwise. 
  We say that $\sigma[r]-\sum_{i=1}^{r-1}\varepsilon_i [i]$ is \emph{partitionable} if an $(n,r,s)$ system with type $\sigma[r]-\sum_{i=1}^{r-1}\varepsilon_i [i]$ exists for
  suitable parameters $n$ and $s$.
\end{Definition}

The notion of type $\sigma[r]-\sum_{i=1}^{r-1}\varepsilon_i [i]$ is motivated by the Solomon--Stiffler construction. E.g.\ $3[7]-1[4]-1[2]$ is 
an abbreviation for the construction taking all points of the ambient space $\PG(6,2)$ three times and subtracting the points of a line and a $4$-dimensional 
subspace. The resulting multiset of points corresponds to a $[3\cdot 127-1\cdot 15-1\cdot 3,7,3\cdot 64-1\cdot 8-1\cdot 2]_2=[363,7,182]_2$ code, which is 
optimal since it attains the Griesmer bound. Here we allow more flexibility by not restricting to $\varepsilon_i\in\{0,1\}$, so that the resulting codes might not 
be distance optimal. On the other hand we restrict the arrangement of the subspaces that are removed from (or added to) a suitable multiple of the ambient space for technical 
reasons. While simplifying the notation and allowing easier statements and proofs, some constructions are not covered by this definition. 

Note that all chains $S_1\subseteq\dots\subseteq S_{r-1}$ are isomorphic, so that the
notion of being partitionable does not depend on the choice of the
subspaces $S_1,\dots,S_{r-1}$. If $\sigma[r]-\sum_{i=1}^{r-1}\varepsilon_i [i]$ is partitionable, then 
also $0[r']-(-\sigma)[r]-\sum_{i=1}^{r-1}\varepsilon_i[i]$ is partitionable for all $r'>r$. 
The parameters of an $(n,r,s)$ system can be computed from the parameters of a partition: 
\begin{Lemma}
  \label{lemma_compute_parameters_from_partition}
  If $\cS$ is an $(n,r,s)$ system with type
  $\sigma[r]-\sum_{i=1}^{r-1}\varepsilon_i [i]$, then we have
  \begin{equation}
    \label{formula_n}
    n=\left(\sigma\cdot \left(2^r-1\right)-\sum_{i=1}^{r-1}\varepsilon_i\cdot\left(2^i-1\right)\right)/3,
    %%\in\N,
  \end{equation}
  \begin{equation}
    \label{eq_s}
    s=\max_{1\le j\le r} \left(s_1 -\sum_{i=1}^{j-1} \varepsilon_i \cdot 2^{i-2}\right),
  \end{equation}
  where
  \begin{equation}
    \label{eq_s1}
    s_1=\left(\sigma\cdot\left(2^{r-2}-1\right)-\sum_{i=2}^{r-1}\varepsilon_i\cdot\left(2^{i-2}-1\right)+\frac{1}{2}\cdot \varepsilon_1\right)/\,3.
  \end{equation}
  Moreover, $\varepsilon_1$ is divisible by $2$ and
  \begin{equation}
    \label{packing_cond}
    \sum_{i=1}^{r-1}\varepsilon_i\cdot\left(2^i-1\right) \equiv \sigma\cdot\left(2^r-1\right)\pmod {3},
  \end{equation}
  where the right hand side is congruent to zero modulo $3$ if $r$ is even.
\end{Lemma}
\begin{proof}
  Let $\cM$ be the multiset of points covered by the elements of $\cS$ and
  $S_1\subseteq\dots\subseteq S_{r-1}$ be subspaces as in
  Definition~\ref{def_partitionable}. Since $\cM$ has cardinality
  $$
    \sigma\cdot\left(2^r-1\right)-\sum_{i=1}^{r-1}\varepsilon_i\cdot\left(2^i-1\right) 
  $$
  and one line contains 3~points, we conclude Equation~(\ref{formula_n}).

  For an arbitrary hyperplane $H$ let $1\le j\le r$ denote the minimal
  integer such that $S_j\not\subseteq H$, where we set $j=r$ if $H=S_{r-1}$. Let $\cM(H)$ denote the number of 
  points of the multiset $\cM$ restricted to hyperplane $H$. When extending the notion $\cM(P)$ of the multiplicity of a point $P$ in 
  a multiset of points $\cM$ additively to arbitrary subspaces $S$ via $\cM(S):=\sum_{P\le S} \cM(P)$, this becomes a special case. 
  Counting points gives
  $$
    \cM(H)= \sigma\cdot\left(2^{r-1}-1\right)-\sum_{i=1}^{j-1}\varepsilon_i\cdot\left(2^i-1\right)-
    \sum_{i=j}^{r-1} \varepsilon_i\cdot\left(2^{i-1}-1\right)
    =\sigma\cdot\left(2^{r-1}-1\right)-\sum_{i=1}^{r-1}\varepsilon_i\cdot\left(2^{i-1}-1\right)
    -\sum_{i=1}^{j-1} \varepsilon_i \cdot 2^{i-1}.
  $$
  The number $s_j$ of elements of $\cS$ contained in $H$ is given by
  $\left(\cM(H)-n)\right)/2$, so that
  \begin{eqnarray*}
    s_j &=& \left( \sigma\cdot\left(2^{r-1}-1\right)-\sum_{i=1}^{r-1}\varepsilon_i\cdot\left(2^{i-1}-1\right)
    -\sum_{i=1}^{j-1} \varepsilon_i\cdot 2^{i-1}
    -\left(\sigma\cdot\left(2^r-1\right)-\sum_{i=1}^{r-1}\varepsilon_i\cdot\left(2^i-1\right)\right)\cdot\frac{1}{3} \right)/2 \\
    &=&  \left(\sigma\cdot\left(2^{r-1}-2\right)-\sum_{i=1}^{r-1} \varepsilon_i\cdot\left(2^{i-1}-2\right) \right)/\,6-\sum_{i=1}^{j-1} \varepsilon_i \cdot 2^{i-2}\\
    &= & \left(\sigma\cdot\left(2^{r-2}-1\right)-\sum_{i=2}^{r-1}\varepsilon_i\cdot\left(2^{i-2}-1\right)+\frac{1}{2}\cdot \varepsilon_1 \right)/\,3-\sum_{i=1}^{j-1} \varepsilon_i\cdot 2^{i-2}.
  \end{eqnarray*}
  This verifies Equation~(\ref{eq_s1}) and yields
  \begin{equation}
    s_j=s_1-\sum_{i=1}^{j-1} \varepsilon_i \cdot 2^{i-2}
  \end{equation}
  for $2\le j\le r$, which implies Equation~(\ref{eq_s}). 
  From $s_2\in\N$ we conclude that $\varepsilon_1$ is divisible by $2$. Equation~(\ref{formula_n})
  implies Equation~(\ref{packing_cond}) and $2^r-1$ is divisible by $3$ iff $r$ is even.
\end{proof}
\begin{Corollary}
  If all $\varepsilon_i$ are nonnegative, then $s=s_1$ (using the notation from Lemma~\ref{lemma_compute_parameters_from_partition}).
\end{Corollary}

\begin{Corollary}
  \label{cor_compute_parameters_from_partition_series}
  If $\cS_t$ is an $\left(n_t,r,s_t\right)$ system with type $\left(\sigma+t\cdot\tfrac{3}{2^{\gcd(r,2)}-1}\right)\cdot [r]-\sum_{i=2}^{r-1}\varepsilon_i [i]$, 
  where $\varepsilon_2,\dots,\varepsilon_{r-1}\in \N$, then we have
  \begin{equation}
    n_t=t\cdot\frac{2^r-1}{2^{\gcd(r,2)}-1}+\left(\sigma\cdot\left(2^r-1\right)-\sum_{i=2}^{r-1}\varepsilon_i\cdot\left(2^i-1\right)\right)/3,
  \end{equation}
  \begin{equation}
     s_t=t\cdot\frac{2^{r-2}-1}{2^{\gcd(r,2)}-1} +\left(\sigma\cdot\left(2^{r-2}-1\right)-\sum_{i=2}^{r-1}\varepsilon_i\cdot\left(2^{i-2}-1\right)\right)/3,
  \end{equation}
  and
  \begin{equation}
     n_t-s_t=    t\cdot\frac{3}{2^{\gcd(r,2)}-1}\cdot 2^{r-2} +
     \sigma\cdot 2^{r-2}-\sum_{i=2}^{r-1}\varepsilon_i\cdot 2^{i-2}.
  \end{equation}
\end{Corollary}

Next we state a few constructions.
\begin{Lemma}
  \label{lemma_partitionable_union}
  If $\sigma[r]-\sum_{i=1}^{r-1}\varepsilon_i[i]$ and $\sigma'[r]-\sum_{i=1}^{r-1}\varepsilon_i'[i]$ are
  partitionable, then $\left(\sigma+\sigma'\right)\cdot[r]-\sum_{i=1}^{r-1}\left(\varepsilon_i+\varepsilon_i'\right)\cdot[i]$ is partitionable.
\end{Lemma}
\begin{proof}
  Fix some subspaces $S_1\subseteq\dots\subseteq S_{r-1}$ as in Definition~\ref{def_partitionable}.
  Let $\cS$ be an $(n,r,s)$ system with type
  $\sigma[r]-\sum_{i=1}^{r-1}\varepsilon_i[i]$ and $\cS'$ be an $(n',r,s')$ system with type
  $\sigma'[r]-\sum_{i=1}^{r-1}\varepsilon_i'[i]$, then the multiset union of
  the elements of $\cS$ and $\cS'$ is an $(n'',r,s'')$ system with type  
  $\left(\sigma+\sigma'\right)\cdot[r]-\sum_{i=1}^{r-1}\left(\varepsilon_i+\varepsilon_i'\right)\cdot[i]$.
\end{proof}

A set of lines that partitions $\PG(r-1,2)$ is called a \emph{line spread}. They do exist iff $r$ is even.
\begin{Lemma}
  \label{lemma_vsp_type}
  For $r>a\ge 2$ with $r\equiv a\pmod 2$ and $\sigma\in\N_{\ge 1}$
  we have that $\sigma[r]-\sigma[a]$ is partitionable.
\end{Lemma}
\begin{proof}
  If $a>2$, then Lemma~\ref{lemma_construction_x_preparation} yields the
  existence of an $(n,r,s)$ system $\cS$ with
  type $[r]-[a]$ and we can use $\sigma$ copies thereof. For $a=2$ we
  replace $\cS$ by a line spread $\PG(r-1,2)$ where we remove
  an arbitrary element.
\end{proof}

\begin{Theorem}
  \label{thm_partition}(Cf.~\cite[page 83]{hirschfeld1998projective},  \cite[Corollary 8]{el2011lambda}, or \cite[Lemma 2]{krotov2023multispreads}) 
  For each $r\ge 2$ we have that $[r]$ is partitionable if $r$ is even and that $3[r]$ is partitionable if $r$ is odd.
\end{Theorem}
\begin{proof}
  The statement is obvious for $r=2$ and for $r=3$ we consider the set of all seven lines in $\PG(2,2)$. From 
  Lemma~\ref{lemma_vsp_type} we deduce that $[r]-[2]$ is partitionable for all even $r\ge 4$ and that $3[r]-3[3]$
  is partitionable for all odd $r\ge 5$, so that the statement follows from Lemma~\ref{lemma_partitionable_union}.
\end{proof}

%% \begin{Lemma}
%%   \label{lemma_partition_1}
%%   For $r\ge 4$ we have that $3\cdot[r]-1\cdot [r-1]-2\cdot [r-3]$ is partitionable.
%% \end{Lemma}
%% \begin{proof}
%%   Let $A$ be an $(r-3)$-space and $B\ge A$ be an $(r-1)$-dimensional subspace of $\PG(r-1,2)$. By $K_1,K_2,K_3$ we denote the 
%%   $(r-2)$-spaces with $A\le K_i\le B$. For $1\le i\le 3$ let $\cV_i$ be a vector space partition of $\PG(r-1,2)$ of type $2^{t_2} (r-2)^1$ where the special 
%%   $(r-2)$-space coincides with $K_i$, see Lemma~\ref{lemma_lifting}. The desired $(n,r,s)$ system is then given by the union of the $t_2$ two-dimensional non-special subspaces of the
%%   three vector space partitions $\cV_i$.
%% \end{proof}

%%\begin{Lemma}
%%  \label{lemma_partition_1g}
%%  For $1\le j\le 2$ and $r\ge 5-j$
%%  \begin{equation}
%%    \left(2^j-1\right)\cdot[r]-1\cdot [r-3+j]-\left(2^{j}-2\right)\cdot [r-3]
%%  \end{equation}
%%  is partitionable.
%%\end{Lemma}
%%\begin{proof}
%%  Let $A$ be an $(r-3)$-space and $B\ge A$ be an $(r-3+j)$-dimensional subspace of $\PG(r-1,2)$. By $K_1,\dots,K_l$ we denote the $l:=2^j-1$
%%  $(r-2)$-spaces with $A\le K_i\le B$. For $1\le i\le l$ let $\cV_i$ be a vector space partition of $\PG(r-1,2)$ of type $2^{t_2} (r-2)^1$ where the special 
%%  $(r-2)$-space coincides with $K_i$, see Lemma~\ref{lemma_lifting}. The desired $(n,r,s)$ system is then given by the union of the $t_2$ two-dimensional non-special subspaces of the
%%  $l$ vector space partitions $\cV_i$.
%%\end{proof}

\begin{Lemma}
  \label{lemma_sigma_constraint}
  If $x[r]-\sum_{i=2}^{r-1}\varepsilon_i [i]$ is partitionable for $x\in\{\sigma,\sigma'\}$ then
  $$\left(\sigma+t\cdot\frac{3}{2^{\gcd(r,2)}-1}\right)\cdot[r]-\sum_{i=2}^{r-1}\varepsilon_i [i]$$ is partitionable for all $t\ge 0$
  and we have $\sigma\equiv \sigma'\pmod {\frac{3}{2^{\gcd(r,2)}-1}}$.
\end{Lemma}
\begin{proof}
  Note that $2^{\gcd(r,2)}-1$ equals $3$ iff $r$ is even and $1$ iff $r$ is odd, so that Theorem~\ref{thm_partition} and Lemma~\ref{lemma_partitionable_union} imply the first statement. 
  For even $r$ the statement $\sigma\equiv \sigma'\pmod {\frac{3}{2^{\gcd(r,2)}-1}}$ is true since $\sigma,\sigma'\in\N$. For odd $r$ we use Equation~(\ref{packing_cond}).
\end{proof}

\begin{Definition}
  We say that $\star[r]-\sum_{i=1}^{r-1}\varepsilon_i [i]$ is
  partitionable if there exists an integer $\sigma$
  such that $\sigma[r]-\sum_{i=1}^{r-1}\varepsilon_i [i]$ is
  partitionable.
\end{Definition}

\begin{Theorem}
  \label{thm_main}
  Let $r\ge 3$, $g:=\gcd(r,2)$, and
  $\varepsilon_2,\dots,\varepsilon_{r-1}\in \mathbb{Z}$ such that
  \begin{equation}
    \label{eq_packing_condition_lemma}
    \sum_{i=2}^{r-1} \varepsilon_i\cdot\left(2^i-1\right) \equiv 0\pmod {2^g-1}.
  \end{equation}
  Then $\star[r]-\sum\limits_{i=2}^{r-1}\varepsilon_i[i]$ is partitionable.
\end{Theorem}
\begin{proof}
  Due to Theorem~\ref{thm_partition} and Lemma~\ref{lemma_partitionable_union} it suffices to consider the 
  case $\varepsilon_i\in\N$ for $2\le i\le r-2$ while Equation~(\ref{eq_packing_condition_lemma}) still holds. From Lemma~\ref{lemma_vsp_type} 
  we conclude that $\varepsilon_i[r]-\varepsilon_i[i]$ is partitionable for all $i\equiv r\pmod 2$ and $2\le i<r$ as well as that 
  $\varepsilon_i[r-1]-\varepsilon_i[i]$ is partitionable for all $i\equiv r-1\pmod 2$ and $2\le i<r-1$. So using Lemma~\ref{lemma_partitionable_union} 
  we can assume $\varepsilon_i=0$ for all $2\le i\le r-2$ while Equation~(\ref{eq_packing_condition_lemma}) still holds. Reusing our first argument again we 
  can additionally assume $\varepsilon_{r-1}\in\N$. If $r$ is odd, then let $\cS$ be a line spread of $S_{r-1}$ (using the notation from Definition~\ref{def_partitionable} and
  the construction from Theorem~\ref{thm_partition}). Choosing each line in $\PG(r-1,2)$ $\varepsilon_{r-1}$ times and removing an $\varepsilon_i$-fold copy 
  of $\cS$ shows that $\star[r]-\varepsilon_{r-1}[r-1]$ is partitionable. If $r$ is even, then Equation~(\ref{eq_packing_condition_lemma}) yields 
  $\varepsilon_{r-1}\equiv 0\pmod 3$. Now let $\cS$ be a multiset of lines partitioning the $3$-fold copy of the points of $S_{r-1}$. 
  Choosing each line in $\PG(r-1,2)$ $\varepsilon_{r-1}/3$ times and removing an $\varepsilon_i/3$-fold copy 
  of $\cS$ shows that $\star[r]-\varepsilon_{r-1}[r-1]$ is partitionable.    
\end{proof}

\begin{Definition}
  For integers $n>s\ge 1$ and $r>2$ let
  the \emph{surplus} be defined by
  \begin{equation}
    \theta(n,r,s):= 3n-g(r,2(n-s)).
  \end{equation}
\end{Definition}
So the surplus is negative iff $n$ is larger than the Griesmer upper
bound for $n_{r/2}(s)$.
\begin{Lemma}
  \label{lemma_asymptotic_construction}
  Let $n>s\ge 1$ and $r>2$ be integers.
  If $\theta(n,r,s)\ge 0$, then there exists an
  $$\left(n+t\cdot \frac{2^r-1}{2^{\gcd(r,2)}-1},r,
  s+t\cdot\frac{2^{r-2}-1}{2^{\gcd(r,2)}-1}\right)$$ system $\cS_t$ for all sufficiently large $t$.
\end{Lemma}
\begin{proof}
  Setting $d':=2(n-s)$ and $n':=g(r,d')$ we have
  $\theta(n,r,s)=3n-n'\ge 0$. Due to Lemma~\ref{lemma_parameters_griesmer_code} we can choose
  integers $\sigma$, $\varepsilon_1,\dots,\varepsilon_{r-1}$, with
  $\sigma\ge 0$ and $0\le \varepsilon_i<2$ for all $1\le i\le r-1$,
  such that
  \begin{equation}
    \label{eq_d_prime_asymptotic_construction}
    d'=\sigma \cdot 2^{r-1}-\sum_{i=1}^{r-1}\varepsilon_i\cdot 2^{i-1}
  \end{equation}
  and
  \begin{equation}
    \label{eq_n_prime_asymptotic_construction}
    n'=\sigma\cdot\left(2^r-1\right)-\sum_{i=1}^{r-1}\varepsilon_i\cdot\left(2^i-1\right).
  \end{equation}
  Since $d'$ is divisible by $2$ we have $\varepsilon_1=0$. Let $\tau:=\theta(n,r,s)$, $\sigma':=\sigma+\tau$,
  $\varepsilon'_{r-1}=\varepsilon_{r-1}+2\tau $, and $\varepsilon_i'=\varepsilon_i$
  for all $1\le i\le r-2$, so that $\varepsilon_i'\in\N$ for all $1\le i\le r-1$
  and $\varepsilon_1'=0$. Note that
  \begin{equation}
    \label{eq_d_prime_asymptotic_construction2}
    d'=\sigma' \cdot 2^{r-1}-\sum_{i=2}^{r-1}\varepsilon_i'\cdot 2^{i-1}
  \end{equation}
  and
  \begin{equation}
    \label{eq_n_prime_asymptotic_construction2}
    3n %%=n'+\tau 
    =\sigma'\cdot\left(2^r-1\right)-\sum_{i=2}^{r-1}\varepsilon_i'\cdot\left(2^i-1\right),
  \end{equation}
  so that
  \begin{equation}
    \label{eq_cond_asymptotic_construction}
    \sum_{i=2}^{r-1}\varepsilon_i'\cdot \left(2^i-1\right)  \equiv 0 \pmod{2^{\gcd(r,2)}-1}.
  \end{equation}
  From Theorem~\ref{thm_main} and Lemma~\ref{lemma_sigma_constraint} we
  conclude that $\left(\sigma'+t\cdot\frac{3}{2^{\gcd(r,2)}-1}\right)\cdot[r]-\sum_{i=2}^{r-1}\varepsilon_i'[i]$ is partitionable for all sufficiently large $t$. 
  From Corollary~\ref{cor_compute_parameters_from_partition_series},
  Equation~(\ref{eq_d_prime_asymptotic_construction2}), and
  Equation~(\ref{eq_n_prime_asymptotic_construction2}) we compute the stated parameters of $\cS_t$.
\end{proof}

\begin{Theorem}
  \label{thm_attained_asymptotically}
  For all sufficiently large $s$ we have that $n_{r/2}(s)$ attains the
  Griesmer upper bound, see Definition~\ref{def_griesmer_uppper_bound}. %% and Inequality~(\ref{ie_griesmer_uppper_bound}).
\end{Theorem}
\begin{proof}
  Let $s_i:=\tfrac{2^{r-2}-1}{2^{\gcd(r,2)}-1} -i$ for $0\le i<
  \tfrac{2^{r-2}-1}{2^{\gcd(r,2)}-1}$ and $n_i$ be the Griesmer upper
  bound for $n_{r/2}(s_i)$, i.e.\ $3n_i\ge
  g(r,2(n_i-s_i))$ while $3\left(n_i+1\right
  )<g(r,2(n_i+1-s_i))$. Let $\sigma_i,
  \varepsilon_{1,i},\dots,\varepsilon_{r-1,i}\in\N$ with $\varepsilon_{j,i}<2$ for
  all $1\le j\le r-1$ be uniquely defined by
  \begin{equation}
    d_i:=2\cdot \left(n_i-s_i\right)=\sigma_i\cdot 2^{r-1}-\sum_{j=1}^{r-1}\varepsilon_{j,i}\cdot 2^{j-1},
  \end{equation}
  so that
  \begin{equation}
    g(r,d_i)=\sigma_i\cdot \left(2^r-1\right)-\sum_{j=1}^{r-1}\varepsilon_{j,i}\cdot\left(2^j-1\right),
  \end{equation}
  using Lemma~\ref{lemma_parameters_griesmer_code}, and $\theta\!\left(n_{i},r,s_{i}\right)\ge 0$. Similarly, let $\sigma_i',
  \varepsilon_{1,i}',\dots,\varepsilon_{r-1,i}'\in\N$ with $\varepsilon_{j,i}'<2$ for
  all $1\le j\le r-1$ be uniquely defined by
  \begin{equation}
    d_i':=2 +d_i=\sigma_i'\cdot 2^{r-1}-\sum_{j=1}^{r-1}\varepsilon_{j,i}'\cdot 2^{j-1},
  \end{equation}
  so that
  \begin{equation}
    g(r,d_i')=\sigma_i'\cdot \left(2^r-1\right)-\sum_{j=1}^{r-1}\varepsilon_{j,i}'\cdot \left(2^j-1\right)
  \end{equation}
  and $\theta\!\left(n_{i}+1,r,s_{i}\right)< 0$.
  Now, let $s_{i,t}:=s_i+t\cdot\tfrac{2^{r-2}-1}{2^{\gcd(r,2)}-1}$ and
  $n_{i,t}:=n_i+t\cdot \tfrac{2^r-1}{2^{\gcd(r,2)}-1}$, so that
  Lemma~\ref{lemma_parameters_griesmer_code} implies
  \begin{equation}
    d_{i,t}:=2\cdot \left(n_{i,t}-s_{i,t}\right)=\left(\sigma_i+t\cdot \frac{3}{2^{\gcd(r,2)}-1}\right)\cdot 2^{r-1}-\sum_{i=1}^{r-i}\varepsilon_i\cdot 2^{i-1},
  \end{equation}
  \begin{equation}
    g\!\left(r,d_{i,t}\right)=t\cdot \frac{2^r-1}{2^{\gcd(r,2)}-1}\cdot 3+
    \sigma_i\cdot \left(2^r-1\right)-\sum_{j=1}^{r-1}\varepsilon_{j,i}\cdot \left(2^j-1\right),
  \end{equation}
  \begin{equation}
    d_{i,t}':=2 +d_{i,t}=\left(\sigma_i'+t\cdot \frac{3}{2^{\gcd(r,2)}-1}\right)\cdot 2^{r-1}-\sum_{i=1}^{r-i}\varepsilon_i'\cdot 2^{i-1},
  \end{equation}
  and
  \begin{equation}
    g\!\left(r,d_{i,t}'\right)=t\cdot \frac{2^r-1}{2^{\gcd(r,2)}-1}\cdot 3+
    \sigma_i'\cdot \left(2^r-1\right)-\sum_{j=1}^{r-1}\varepsilon_{j,i}'\cdot \left(2^j-1\right).
  \end{equation}
  Thus we have $$\theta\!\left(n_{i,t},r,s_{i,t}\right)=\theta\!\left(n_{i},r,s_{i}\right)\ge 0\,\,\text{and}\,\,\theta\!\left(n_{i,t}+1,r,s_{i,t}\right)=\theta\!\left(n_{i}+1,r,s_{i}\right)<0,$$
  i.e.\ the Griesmer upper bound for $n_{r/2}\!\left(s_{i,t}\right)$ is given
  by $n_{i,t}$ for all $t\in \N$. Lemma~\ref{lemma_asymptotic_construction} yields the existence of an $\left(n_{i,t},r,s_{i,t}\right)$ system $\cS_{i,t}$ for all sufficiently large $t$.
\end{proof}
\begin{Corollary}
  \label{cor_attained_asymptotically}
  Assuming a sufficiently large $d\in\mathbb{N}$ and $2k\in\mathbb{N}_{\ge 3}$, an additive quaternary block code $C\subseteq \F_4^n$ 
  with $|C|=4^k$ and minimum Hamming distance $d$ exists iff $n\ge \left\lceil \frac{g(2k,2d)}{3}\right\rceil
  =\left\lceil \frac{\sum_{i=0}^{2k-1}\left\lceil\frac{2d}{2^i}\right\rceil}{3}\right\rceil$. 
\end{Corollary}
We remark that the statement of Theorem~\ref{thm_attained_asymptotically} was generalized to arbitrary additive codes over $\F_q$ in \cite{kurz2024additive}. 
The proof of Theorem~\ref{thm_attained_asymptotically} suggests the following algorithm to determine explicit formulas for $n_{r/2}(s)$ assuming that $s$
is sufficiently large. For all $0\le i< \tfrac{2^{r-2}-1}{2^{\gcd(r,2)}-1}$ compute
the Griesmer upper bound $n_i$ for $n_{r/2}(s_i)$ where $s_i=\tfrac{2^{r-2}-1}{2^{\gcd(r,2)}-1} -i$. Then we have
\begin{equation}
  n_{r/2}\!\left(t\cdot \frac{2^{r-2}-1}{2^{\gcd(r,2)}-1} -i \right)=
  t\cdot \frac{2^r-1}{2^{\gcd(r,2)}-1} -\left(\frac{2^r-1}{2^{\gcd(r,2)}-1}-n_i\right)
\end{equation}
for all sufficiently large $t$. As an example we mention:
\begin{Proposition} (Cf.~\cite[Table I]{guan2023some},\cite[Table II]{10693309})
  \label{prop_q_2_h_2_r_7}
  For all sufficiently large $t$ we have
  \begin{itemize}
    \item $n_{3.5}(31t)=127t$;
    \item $n_{3.5}(31t-1)=127t-5$;
    \item $n_{3.5}(31t-2)=127t-10$;
    \item $n_{3.5}(31t-3)=127t-15$;
    \item $n_{3.5}(31t-4)=127t-20$;
    \item $n_{3.5}(31t-5)=127t-21$;
    \item $n_{3.5}(31t-6)=127t-26$;
    \item $n_{3.5}(31t-7)=127t-31$;
    \item $n_{3.5}(31t-8)=127t-36$;
    \item $n_{3.5}(31t-9)=127t-41$;
    \item $n_{3.5}(31t-10)=127t-42$;
    \item $n_{3.5}(31t-11)=127t-47$;
    \item $n_{3.5}(31t-12)=127t-52$;
    \item $n_{3.5}(31t-13)=127t-55$;
    \item $n_{3.5}(31t-14)=127t-60$;
    \item $n_{3.5}(31t-15)=127t-63$;
    \item $n_{3.5}(31t-16)=127t-68$;
    \item $n_{3.5}(31t-17)=127t-73$;
    \item $n_{3.5}(31t-18)=127t-76$;
    \item $n_{3.5}(31t-19)=127t-81$;
    \item $n_{3.5}(31t-20)=127t-84$;
    \item $n_{3.5}(31t-21)=127t-87$;
    \item $n_{3.5}(31t-22)=127t-92$;
    \item $n_{3.5}(31t-23)=127t-95$;
    \item $n_{3.5}(31t-24)=127t-100$;
    \item $n_{3.5}(31t-25)=127t-105$;
    \item $n_{3.5}(31t-26)=127t-108$;
    \item $n_{3.5}(31t-27)=127t-113$;
    \item $n_{3.5}(31t-28)=127t-116$;
    \item $n_{3.5}(31t-29)=127t-121$;
    \item $n_{3.5}(31t-30)=127t-126$.
  \end{itemize}
\end{Proposition}
In \cite{10693309} the stated formulas of Proposition~\ref{prop_q_2_h_2_r_7}
were indeed shown to be true for all
$t\ge 2$ and $n_2(7,2;31-i)$ was determined for all
$i\in\{0,\dots,31\}\backslash\{19,24,25\}$, referring to \cite{kurz2024computer} for $i=18$ and \cite{guan2023some} for the previous state of the art.

\section{Exact values of $n_{3.5}(s)$ and $n_4(s)$}
\label{sec_exact}

The existence of an $(n,r,s)$ systems can be easily modeled as ILP (Integer Linear Programming) problem. Denoting the set of lines in $\PG(r-1,2)$ by $\mathcal{L}$ and 
the set of hyperplanes in $\PG(r-1,2)$ by $\mathcal{H}$, an $(n,r,s)$ system exists iff the ILP  
\begin{eqnarray*}
  \sum_{L\in\mathcal{L}} x_L=n\\ 
  \sum_{L\le H} x_L \le s \quad\forall H\in\mathcal{H}\\ 
  x_L\in\mathbb{N}\quad\forall L\in\mathcal{L}
\end{eqnarray*}
admits a solution. To reduce the search space we prescribe subgroups of the automorphism group.
%% write_ILP_jozefien43.cpp and write_ILP_jozefien4.cpp
Alternatively we can try to partition suitable multisets of points into lines.
%% write_ILP_jozefien63.cpp and write_ILP_jozefien6.cpp
Those multisets of points can again be modeled as ILP problems and we may prescribe
subgroups of the automorphism group, see e.g.\ \cite{braun2005optimal}.
%% write_ILP_jozefien53.cpp and write_ILP_jozefien5.cpp
%%  // Magma online: http://magma.maths.usyd.edu.au/calc/
%%  // G:= GL(8, 2);
%%  // a:=Classes(G);
%%  // for i:=21 to 40 do
%%  // print a[i];
%%  // end for;
Alternatively we use the database of \emph{best known linear codes} (BKLC)
in \texttt{Magma} or enumerate suitable linear codes using \texttt{LinCode}
\cite{bouyukliev2021computer}. For the (known) conditions of the binary codes 
we refer to Lemma~\ref{lemma_binary_code}.  

\begin{Theorem} (Cf.~\cite[Table I]{guan2023some},\cite[Table II]{10693309})
  We have
  \begin{itemize}
    \item $n_{3.5}(31t)=127t$ for $t\ge 1$;
    \item $n_{3.5}(31t-1)=127t-5$ for $t\ge 1$;
    \item $n_{3.5}(31t-2)=127t-10$ for $t\ge 1$;
    \item $n_{3.5}(31t-3)=127t-15$ for $t\ge 1$;
    \item $n_{3.5}(31t-4)=127t-20$ for $t\ge 1$;
    \item $n_{3.5}(31t-5)=127t-21$ for $t\ge 1$;
    \item $n_{3.5}(31t-6)=127t-26$ for $t\ge 1$;
    \item $n_{3.5}(31t-7)=127t-31$ for $t\ge 1$;
    \item $n_{3.5}(31t-8)=127t-36$ for $t\ge 1$;
    \item $n_{3.5}(31t-9)=127t-41$ for $t\ge 1$;
    \item $n_{3.5}(31t-10)=127t-42$ for $t\ge 1$;
    \item $n_{3.5}(31t-11)=127t-47$ for $t\ge 1$;
    \item $n_{3.5}(31t-12)=127t-52$ for $t\ge 1$;
    \item $n_{3.5}(31t-13)=127t-55$ for $t\ge 1$;
    \item $n_{3.5}(31t-14)=127t-60$ for $t\ge 1$;
    \item $n_{3.5}(31t-15)=127t-63$ for $t\ge 1$;
    \item $n_{3.5}(31t-16)=127t-68$ for $t\ge 1$;
    \item $n_{3.5}(31t-17)=127t-73$ for $t\ge 1$;
    \item $n_{3.5}(31t-18)=127t-76$ for $t\ge 1$;
    \item $n_{3.5}(31t-19)=127t-81$ for $t\ge 1$;
    \item $n_{3.5}(31t-20)=127t-84$ for $t\ge 1$;
    \item $n_{3.5}(31t-21)=127t-87$ for $t\ge 1$;
    \item $n_{3.5}(31t-22)=127t-92$ for $t\ge 1$;
    \item $n_{3.5}(31t-23)=127t-95$ for $t\ge 1$;
    \item $n_{3.5}(31t-24)=127t-100$ for $t\ge 1$;
    \item $n_{3.5}(31t-25)=127t-105$ for $t\ge 1$;
    \item $n_{3.5}(31t-26)=127t-108$ for $t\ge 2$ and $n_{3.5}(5)=17$;
    \item $n_{3.5}(31t-27)=127t-113$ for $t\ge 2$ and $n_{3.5}(4)=12$;
    \item $n_{3.5}(31t-28)=127t-116$ for $t\ge 2$ and $n_{3.5}(3)=7$;
    \item $n_{3.5}(31t-29)=127t-121$ for $t\ge 2$;
    \item $n_{3.5}(31t-30)=127t-126$ for $t\ge 2$.
  \end{itemize}
\end{Theorem}
\begin{proof}
  We can assume $s\ge 3$. %% Lemma~\ref{lemma_small} gives $n_2(7,2;1)=1$ and $n_2(7,2;2)=2$. 
  Theorem~\ref{thm_partition} yields $n_2(7,2;31t)=127t$ for $t\ge 1$.
  In \cite{blokhuis2004small} $n_2(7,2;3)\le 7$ was shown. The coding
  upper bound implies $n_2(7,2;4)\le 12$ and $n_2(7,2;5)\le 17$. All other
  upper bounds follow from the Griesmer upper bound. Due to Theorem~\ref{thm_partition} and Lemma~\ref{lemma_union}
  it suffices to give constructions for
  $
    s\in \{3,\dots, 13, 15, 21, 25, 26, 30\}
  $. 
  Using ILP searches we have found the following explicit constructions:\footnote{We remark that constructions for $s=3,4$ were given in \cite{blokhuis2004small} and
  for $s=5,21$ we can use quaternary linear codes. For $s=9$ an example is given by a vector space partition of $\PG(6,2)$ of type $2^{35} 3^1 4^1$.
  For $s=15$ an example is given in \cite[Example 2]{10693309}. More constructions can e.g.\ be found in \cite{10693309}.}\\  
  $s=3:$
  $\left(
\right)$.

\end{proof}

\begin{Theorem}
\label{thm_n_4_s}
  For $s\ge 30$ the Griesmer upper bound for $n_4(s)$ can always be attained.
  \begin{itemize}
    \item $n_4(21t)=85t$ for $t\ge 1$;
    \item $n_4(21t-1)=85t-5$ for $t\ge 1$;
    \item $n_4(21t-2)=85t-10$ for $t\ge 1$;
    \item $n_4(21t-3)=85t-15$ for $t\ge 1$;
    \item $n_4(21t-4)=85t-20$ for $t\ge 1$;
    \item $n_4(21t-5)=85t-21$ for $t\ge 1$;
    \item $n_4(21t-6)=85t-26$ for $t\ge 2$ and $n_4(15)=55$; %% $n_4(15)\in\{55,56,57\}$;
    \item $n_4(21t-7)=85t-31$ for $t\ge 1$;
    \item $n_4(21t-8)=85t-36$ for $t\ge 1$;
    \item $n_4(21t-9)=85t-41$ for $t\ge 1$;
    \item $n_4(21t-10)=85t-42$ for $t\ge 2$ and $n_4(11)=40$;
    \item $n_4(21t-11)=85t-47$ for $t\ge 2$ and $n_4(10)=36$;
    \item $n_4(21t-12)=85t-52$ for $t\ge 1$;
    \item $n_4(21t-13)=85t-55$ for $t\ge 3$, $n_4(8)=28$, and $n_4(29)=113$; %% $n_4(29)\in\{113,114,115\}$;
    \item $n_4(21t-14)=85t-60$ for $t\ge 3$, $n_4(7)=23$, and $n_4(28)=108$; %% $n_4(28)\in\{108,109,110\}$;
    \item $n_4(21t-15)=85t-63$ for $t\ge 2$ and $n_4(6)=18$;
    \item $n_4(21t-16)=85t-68$ for $t\ge 1$;
    \item $n_4(21t-17)=85t-73$ for $t\ge 2$ and $n_4(4)=10$;
    \item $n_4(21t-18)=85t-76$ for $t\ge 2$ and $n_4(3)=5$;
    \item $n_4(21t-19)=85t-81$ for $t\ge 2$;
    \item $n_4(21t-20)=85t-84$ for $t\ge 2$.
  \end{itemize}
\end{Theorem}
\begin{proof}
  We can assume $s\ge 3$. %% Lemma~\ref{lemma_small} gives $n_4(1)=1$ and $n_4(2)=2$. 
  Theorem~\ref{thm_partition} yields $n_4(21t)=85t$ for $t\ge 1$.
  In \cite{blokhuis2004small} $n_4(3)\le 5$ and $n_4(4)\le 10$ were shown. The weak coding upper bound implies $n_4(6)\le 18$,
  $n_4(7)\le 23$, $n_4(8)\le 28$, $n_4(10)\le 36$, and $n_4(11)\le 40$. The strong coding upper bound implies $n_4(15)\le 55$, 
  $n_4(28)\le 108$, and $n_4(29)\le 113$, see Subsection~\ref{sec_nonexistence}. All other
  upper bounds follow from the Griesmer upper bound. For all $$s\in\{5,\dots,60\}\backslash\{9,10,11,14,15,23,24,27,28,29,44,45,49,50\}$$ 
  the mentioned upper bound for $n_4(s)$ is matched by a quaternary linear code. We have $n_4(44)\ge n_4(21)+n_4(23)=85+89=174$ and $n_4(45)\ge n_4(21)+n_4(24)=85+94=179$. 
  The following explicit examples were obtained using ILP searches:\\ 
  
\noindent
$s=9$: 
$\left(
\right)$.

  With this, all remaining constructions can be obtained using Theorem~\ref{thm_partition} and Lemma~\ref{lemma_union}.
\end{proof}

In Table~\ref{table_dim_4} we state the known bounds for $n_4(s)$ when $s\le 60$. Lower bounds based on quaternary linear codes are stated in columns headed with {\lq\lq}L{\rq\rq}.
Upper bounds, based on \cite{blokhuis2004small} for $s\le 4$ and on binary linear codes for $s>4$, i.e.\ the strong coding upper bound, are stated in columns headed with {\lq\lq}U{\rq\rq}. 
Actually, for $s\ge 5$ the weak coding upper bound is sufficient in all but three cases, which we indicate in bold face and discuss in detail in Subsection~\ref{sec_nonexistence}.
Values of improved constructions are given in columns headed with {\lq\lq}I{\rq\rq}. %% Open cases are marked in bold font and we 
We remark that we have $n_4(s)=n_4(s-21)+85$ for $n>60$. 
For $s>60$ there are improvements over the linear case iff $s$ is congruent to $2$, $3$, $7$, or $8$ modulo $21$.  
Generator matrices of the improvements are given in the proof of Theorem~\ref{thm_n_4_s}. We observe that $n_4(44)\ge n_4(23)+n_4(21)$ and $n_4(45)\ge n_4(24)+n_4(21)$  are attained with equality.
%% and that there are easy geometric constructions for $s\in \{49,50\}$. 

%% Choosing $l=5$ for $k=4$ yields a multiset $\mathcal{L}^\star$ of $160$ lines with $s=40$. Consider hyperplane $H$ of the 
%% construction as $\PG(6,2)$ and insert the lines from a vector space partition of type $2^{32}5^1$. This yields a multiset of $192$ lines with $s=48$. 
%% Now consider the special subspace $A$ of the construction as $\PG(4,2)$ and insert either three lines in a three-dimensional subspace or the lines from a vector space partition 
%% of type $2^{8}3^1$. This yields examples for $s\in\{49,50\}$.\footnote{Actually, the constructions from Section~\ref{sec_geometric_constructions} are sufficient to attain the maximal 
%% number $n_k^q(s)$ of lines in $\PG(2k-1,q)$ such that at most $s$ lines are contained in a hyperplane, assuming that $s$ is sufficiently large. I.e., the upper bound implied 
%% by the Griesmer bound can always be attained if $s$ is sufficiently large, c.f.~\cite{bierbrauer2024asymptotic}. We do not know what happens if we replace lines by 
%% subspaces with a larger dimension, see e.g.\ \cite{ball2024additive}.}    

\begin{table}[htp]
  \begin{center}
    \begin{tabular}{rrrr|rrrr|rrrr}
      \hline
      s & L & I & U & s & L & I & U & s & L & I & U \\ 
      \hline
       1 & -- &    & -- & 21 &  85 &     &  85 & 41 & 165 & & 165\\
       2 & -- &    & -- & 22 &  86 &     &  86 & 42 & 170 & & 170 \\
       3 &  5 &    &  5 & 23 &  87 &  89 &  89 & 43 & 171 & & 171 \\
       4 & 10 &    & 10 & 24 &  92 &  94 &  94 & 44 & 172 & 174 & 174 \\
       5 & 17 &    & 17 & 25 &  97 &     &  97 & 45 & 177 & 179 & 179 \\
       6 & 18 &    & 18 & 26 & 102 &     & 102 & 46 & 182 & & 182 \\
       7 & 23 &    & 23 & 27 & 103 & 107 & 107 & 47 & 187 & & 187 \\
       8 & 28 &    & 28 & 28 & 108 &     & \textbf{108} & 48 & 192 & & 192 \\
       9 & 31 & 33 & 33 & 29 & 113 &     & \textbf{113} & 49 & 193 & 195 & 195 \\
      10 & 34 & 36 & 36 & 30 & 118 &     & 118 & 50 & 198 & 200 & 200 \\
      11 & 39 & 40 & 40 & 31 & 123 &     & 123 & 51 & 203 & & 203\\ 
      12 & 44 &    & 44 & 32 & 128 &     & 128 & 52 & 208 & & 208 \\
      13 & 49 &    & 49 & 33 & 129 &     & 129 & 53 & 213 & & 213 \\
      14 & 50 & 54 & 54 & 34 & 134 &     & 134 & 54 & 214 & & 214 \\ 
      15 & 55 &    & \textbf{55} & 35 & 139 &     & 139 & 55 & 219 && 219\\ 
      16 & 64 &    & 64 & 36 & 144 &     & 144 & 56 & 224 & & 224 \\
      17 & 65 &    & 65 & 37 & 149 &     & 149 & 57 & 229 & & 229 \\  
      18 & 70 &    & 70 & 38 & 150 &     & 150 & 58 & 234 & & 234 \\
      19 & 75 &    & 75 & 39 & 155 &     & 155 & 59 & 235 & & 235 \\ 
      20 & 80 &    & 80 & 40 & 160 &     & 160 & 60 & 240 & & 240 \\
      \hline
    \end{tabular}
    \caption{Bounds for $n_4(s)$.}
    \label{table_dim_4}
  \end{center}
\end{table}  

\subsection{Non-existence of binary linear codes related to quaternary additive codes}
\label{sec_nonexistence}

In this subsection we want to evaluate the strong coding upper bound for $n_4(15)$, $n_4(28)$, and  $n_4(29)$, i.e.\ the three values in Table~\ref{table_dim_4} 
which have a bold entry in the {\lq\lq}U{\rq\rq}-column.\footnote{For $n_4(3)$ and $n_4(4)$ we state that there exist $217$ non-isomorphic even $[18,7,\{6,\dots,12\}]_2$ and 
more than $10^6$ non-isomorphic even $[33,7,\{14,\dots,22\}]_2$ codes, i.e.\ the strong coding upper bound does not imply the exact values.} 
It will turn out that the weak coding upper bound is improved by two in those three cases. The strong coding 
upper bound is based on Lemma~\ref{lemma_binary_code}, i.e.\ we have to show the non-existence of certain binary linear codes taking an upper bound on the maximum weight 
into account. To this end we denote an $[n,k]_2$ code with non-zero weights of the codewords contained in $W\subsetneq \mathbb{N}$ as an $[n,k,W]_2$ code and 
use the software package \texttt{LinCode}\cite{bouyukliev2021computer} for the exhaustive enumeration of linear codes. More precisely, we 
will recursively construct $[n,k,W]_2$ codes from $[n',k-1,W]_2$ codes for some suitable set $W\subsetneq \mathbb{N}$ of possible non-zero weights. Initially we will start 
with the unique $[d,1,d]_2$ code for the desired minimum distance $d$. Given some $[n',k-1,W]_2$ code $C$ we check that they cannot be extended to $[n',k,W]_2$ codes. Let 
the parameters of the desired final code $\tilde{C}$ be denoted by $[\tilde{n},\tilde{k},W]_2$. The residual code of $\tilde{C}$ with respect to the support of $C$ is an 
$[\tilde{n}-n',\tilde{k}-k+1,d']_2$. If $\hat{d}$ is an upper bound on $d'$, then it suffices to consider the case $n\le n+d'$. In Table~\ref{table_330_8_164} 
we state the counts for the intermediate codes aiming at a $[330,8,164]_2$ code with maximum weight at most $220$. Since no $[330,8,165]_2$ code exists, we can assume 
the existence of a codeword of weight $164$. The minimum distance of a residual $[166,7]_2$ code is at most $82$, so that it suffices to consider $[\le 246,2,W]_2$ codes. 
For a $[246,2,W]_2$ code the residual code is a $[84,6]_2$ code with minimum distance at most $41$, 
for a $[287,3,W]_2$ code the residual code is a $[43,5]_2$ code with minimum distance at most $21$, for a $[308,4,W]_2$ code the residual code is a $[22,4]_2$ code 
with minimum distance at most $11$, for a $[319,5,W]_2$ code the residual code is a $[11,3]_2$ code with minimum distance at most $6$, and for a $[325,6,W]_2$ code 
the residual code is a $[5,2]_2$ code with minimum distance at most $2$. Note that the possible minimum lengths of $[n,k,164]_2$ codes are indeed 
given by $164$, $246$, $287$, $308$, $319$, $325$, $328$, and $330$ for $1\le k\le 8$, all given by the Griesmer bound.                

\begin{table}[htp]
  \begin{center}
    \begin{tabular}{rrrrrrrrrrr}
      \hline
      n  & 164 & 246 & 287 & 308 & 319 &  325 & 328 & 330 \\  
      k  &   1 &   2 &   3 &   4 &   5 &    6 &   7 &   8 \\
      \# &   1 &   1 &   1 &   1 &   3 &   34 &  69 &   0 \\ 
      \hline
    \end{tabular}     
  \end{center}  
  \caption{Number of $[n,k,\{164,168,\dots,220\}\backslash\{196\}]_2$ codes which are recursively obtainable.}
  \label{table_330_8_164}
\end{table}

%% \begin{table}[htp]
%%   \begin{center}
%%     \begin{tabular}{rrrrrrrrrrrrrrrr}
%%       \hline
%%       n  & 164 & 246 & 248 & 250 & 287 & 288 & 289 & 290 & 308 & 309 & 310 & 319 & 320 & 325 & 328 \\  
%%       k  &   1 &   2 &   2 &   2 &   3 &   3 &   3 &   3 &   4 &   4 &   4 &   5 &   5 &   6 &   7 \\
%%       \# &   1 &   1 &   1 &   2 &   1 &   1 &   2 &   2 &   1 &   2 &   8 &   3 &  28 &  34 &  69 \\ 
%%       \hline
%%     \end{tabular}     
%%   \end{center}  
%%   \label{table_330_8_164}
%%   \caption{Number of $[n,k,\{164,168,\dots,220\}\backslash\{196\}]_2$ codes which are recursively obtainable.}
%% \end{table}  

While $[330,8,164]_2$ and $[345,8,172]_2$ codes can be easily constructed from the Solom--Stiffler construction, the resulting maximum weights violate  
Lemma~\ref{lemma_binary_code}. In Lemma~\ref{lemma_no_330_8_164} and Lemma~\ref{lemma_no_345_8_172} we show that this indeed the case in general. 

\begin{Lemma}
  \label{lemma_no_330_8_164}
  No $[330,8,164]_2$ code with maximum weight at most $220$ exists.
\end{Lemma}
\begin{Proof}
  Assume that $C$ is a $[330,8,164]_2$ code with maximum weight at most $220$. From Lemma~\ref{lemma_griesmer_divisibility} we conclude that $C$ is $4$-divisible. From 
  Lemma~\ref{lemma_residual_code} and the non-existence of the corresponding residual codes we conclude that $C$ does not have a codeword with weight $196$. 
  So, consider the set $W=\{164,168,\dots,220\}\backslash\{196\}$ of possible non-zero weights. Using \texttt{LinCode}\cite{bouyukliev2021computer} we have iteratively 
  constructed $[n,k,W]_2$ codes as stated in Table~\ref{table_330_8_164}. 
\end{Proof}
\begin{Corollary} $n_4(28)\le 109$.\end{Corollary} %% Solom-Stiffler: 2[8]-[7]-[5]-[4]-[3]

\begin{table}[htp]
  \begin{center}
    \begin{tabular}{rrrrrrrrrrr}
      \hline
      n  & 172 & 258 & 301 & 323 & 334 &  340 & 343 & 345 \\  
      k  &   1 &   2 &   3 &   4 &   5 &    6 &   7 &   8 \\
      \# &   1 &   1 &   1 &   1 &   1 &    2 &   3 &   0 \\ 
      \hline
    \end{tabular}     
  \end{center}  
  \caption{Number of $[n,k,\{172,176,184,188,192,200,216,220,224,228\}]_2$ codes which are recursively obtainable.}
  \label{table_345_8_172}
\end{table}

\begin{Lemma}
  \label{lemma_no_345_8_172}
  No $[345,8,172]_2$ code with maximum weight at most $230$ exists.
\end{Lemma}
\begin{Proof}
  Assume that $C$ is $[345,8,172]_2$ code with maximum weight at most $228$. From Lemma~\ref{lemma_griesmer_divisibility} we conclude that $C$ is $4$-divisible. From 
  Lemma~\ref{lemma_residual_code} and the non-existence of the corresponding residual codes we conclude that $C$ does not have a codeword with a weight in $\{180,196,204,208,212,230\}$. 
  So, consider the set $W:=\{172,176,184,188,192,200,216,220,224,228\}$ of possible non-zero weights. Using \texttt{LinCode}\cite{bouyukliev2021computer} we have iteratively 
  constructed $[n,k,W]_2$ codes as stated in Table~\ref{table_345_8_172}.
\end{Proof}

\begin{Corollary} $n_4(29)\le 114$.\end{Corollary} %% Solom-Stiffler: 2[8]-[7]-[5]-[3]

While the computation underlying Lemma~\ref{lemma_no_345_8_172} took just a few minutes, for Lemma~\ref{lemma_no_330_8_164} forty-four hours were needed. In principle 
the non-existence of an even $[171,8,84]_2$ code with maximum weight at most $114$, i.e., the binary code corresponding to a  $(57,8,15)$ system, can be obtained using the
same computational approach. However, the required computation time would be significantly large, which is partially due to the fact that the cannot deduce 
$4$-divisibility\footnote{$4$-divisibility can be concluded nevertheless: By a residual code argument the only possible 
weights not divisible by four can be reduced to $\{90,106,110\}$. From Proposition~\ref{prop_div_one_more} with $a=1$ we can conclude $T\in\{192,256\}$ using the ILP method 
and the existence of a $4$-divisible subcode with dimension at least $7$. Applying the ILP method for the corresponding split weight enumerator, 
see e.g.\ \cite{jaffe1997brief,simonis1995macwilliams} for details, those three weights can be excluded relatively easily.} from Lemma~\ref{lemma_residual_code} and, 
more importantly, since $g(8,84)=170<171$. The known optimal $[171,8,84]_2$ code contains a codeword of weight $128$ and two codewords of 
weight $100$.\footnote{Note that the residual code of a codeword of weight $100$ is a unique $[71,7,34]_2$ code, see e.g.~\cite{bouyukliev2001optimal}, and especially contains 
a codeword of weight $64$. This is impossible given a maximum weight of $114$ in the original code.} 

Slightly enhancing our computational approach we next show the non-existence of a $[168,8,82]_2$ code with maximum weight $112$. By Lemma~\ref{lemma_binary_code}   
no $(56,8,15)$ system exists. Thus, also no $(55,8,15)$ system can exist. We can also deduce the slightly stronger result of the non-existence of a 
$[171,8,84]_2$ code $C$ with maximum weight at most $114$ by using a parity bit argument and by deducing the existence of a full line in the support of $C$.

\begin{Definition}
  Let $C$ be an $[n,k]_2$ code and $\cM_k$ be its corresponding multiset of points in $\PG(k-1,2)$. We recursively choose points $P_i$ and construct $\cM_i$ from $\cM_{i+1}$ by projection 
  through $P_i$. We call $\big(\left|\cM_{k}(P_k)\right|,\dots,\left|\cM_1(P_1)\right|\big)$ the extensions vector with respect to the points $P_k,\dots,P_1$. The lexicographic maximum over all choices 
  is called the (canonical) extension vector of $\cM$. 
\end{Definition}
Clearly, the points $P_k,\dots,P_1$ of the canonical extension vector have to be endpoints of $\cM_i$, i.e.\ points with the maximum multiplicity.

\begin{Remark}
  For a multiset of points $\cM$ in $\PG(k-1,q)$ the maximum multiplicity of an $i$-dimensional subspace $S_i$ is usually denoted by $\gamma_i(\cM)$. I.e.\ $\gamma_1(\cM)$ is 
  the maximum point multiplicity and $\gamma_k(\cM)=|\cM|$. As a convention we set $\gamma_0(\cM)=0$. If $\cM$ is spanning and $C$ denotes the linear code corresponding to $\cM$, then 
  $|\cM|-\gamma_i(\cM)$ equals the maximum possible minimum support weight of q $(k-i)$-dimensional subcodes of $C$. Note that 
  $\left(\gamma_1(\cM)-\gamma_0(\cM),\dots,\gamma_k(\cM)-\gamma_{k-1}(\cM)\right)$ might be lexicographically larger than the canonical extension vector of $\cM$ since the subspaces 
  $S_i=\left\langle P_k,\dots,P_{k-i+1}\right\rangle$ have to form a chain $S_1\subsetneq \dots\subsetneq S_k$ for the latter case  
\end{Remark}

\begin{Lemma}
  \label{lemma_no_168_8_82}
  No $[168,8,82]_2$ code with maximum weight at most $112$ exists.
\end{Lemma}
\begin{Proof}
  Assume that $C'$ is a $[168,8,82]_2$ code with maximum weight at most $112$. Let $C$ arise from $C'$ by shortening one coordinate and adding 
  a parity bit. So, $C$ is an even $[168,8,82]_2$ code and we consider a residual code w.r.t.\ a codeword of weight $w$.
  \begin{itemize}
    \item For $w=82$ the residual code is an $[86,7,41]_2$ code with non-zero weights in\\ $\{41,42,43,44,45,46,47,48,49,50,51,52,53,54,57,58,61,62,63,64\}$.
    \item For $w=90$ the residual code is a $[78,7,37]_2$ code with non-zero weights in\\ $\{37,38,39,40,41,42,43,44,45,46,47,48,53,54,55,56,61,62,63,64\}$.
    \item For $w=94$ the residual code is a $[74,7,35]_2$ code with non-zero weights in\\ $\{35,36,37,38,39,40,41,42,43,44,47,48,51,52,55,56,59,60,63,64,67,68\}$.
    \item For $w=98$ the residual code is a $[70,7,33]_2$ code with non-zero weights in\\ $\{33,34,35,36,37,38,63,64\}$.
    \item For $w=110$ the residual code is a $[58,7,27]_2$ code with non-zero weights in\\ $\{27,28,31,32,35,36,39,40,43,44,47,48,51,52\}$.
    \item For $w=112$ the residual code is a $[56,7,26]_2$ code with non-zero weights in\\ $\{26,28,30,32,34,36,38,40,42,44,46,48,50,52,54,56\}$. 
  \end{itemize}
  Here we use the classification of optimal $[n,\le 7]_2$ codes from \cite{bouyukliev2001optimal} and shortening if the minimum distance is odd.
  Since no $[165,7,82]_2$ code exists the maximum column multiplicity of $C$ is at most $2$. We have verified by exhaustive enumeration that no 
  $2$-divisible $[166,7,\{82,\dots,112\}]_2$ code exists. (The known optimal $[166,7,82]_2$ code contains a codeword of weight $128$.) Thus, $C$ has to be projective as 
  do all residual codes of $C$.
  
  Let $W:=\{82,84,\dots,112\}$. Starting from an $[82,1,W]_2$ code we construct extensions to $[146,2,W]_2$ codes and then recursively 
  consider all extensions that may lead to a $[168,8,W]_2$ code. During the generation with \texttt{LinCode} we remove codes that contain a residual code with a forbidden 
  weight, as specified above. No $[168,8,W]_2$ code was reached, see Table~\ref{table_168_8_82_82_64} for the counting details.     
  %% 64^1: [1,82] #=1, [2,146] #=3, [3,<=154] #=0, [3,155] #=2, [4,<=159] #=0, [4,160] #=2, [5,<=162] #=0, [5,163] #=2, [5,164] #=43, [6,<=165] #=0, [6,166] #=74, [7,167] #=0   
  So, in the following we can assume that each residual code of a codeword of weight $82$ does not contain a codeword of weight $64$ and we continue to exclude further weights.
  To this end we refine the extension strategy a bit. Let $c,c'\in C$ be two different codewords in $C$ where $c$ has weight $82$. Let $\tilde{c}$ be the codeword in the residual
  code of $c$ that corresponds to $c'$, i.e.\ the residual code with respect to the support of $c$ and $c'$ is a projective $[168-\operatorname{wt}(c)-\operatorname{wt}(\tilde{c}),6,d]_2$ 
  code $C'$. By enumerating all even $[\le 168,3,\{82,\dots,112\}]_2$ codes we have verified that for e.g.\ $\operatorname{wt}(\tilde{c})\in\{45,49,57,58,61,62,63,64\}$, $C'$ has to be a distance
  optimal code. Counts per canonical extension vector are given as follows:
  \begin{itemize}
    \item $\operatorname{wt}(\tilde{c})=64$, $[22,6,9]_2$: $(1,\!2,\!2,\!3,\!5,\!9)\to 184$, $(1,\!1,\!2,\!4,\!5,\!9)\to 1$;
    \item $\operatorname{wt}(\tilde{c})=63$, $[23,6,10]_2$: $(1,\!2,\!2,\!3,\!5,\!10)\to 21$;
    \item $\operatorname{wt}(\tilde{c})=62$, $[24,6,10]_2$: $(1,\!2,\!3,\!3,\!5,\!10)\!\to\! 2064$, $(1,\!2,\!2,\!4,\!5,\!10)\!\to\! 74$, $(1,\!2,\!2,\!3,\!6,\!10)\!\to\! 11$, $(1,\!1,\!2,\!4,\!6,\!10)\!\to\! 2$;
    \item $\operatorname{wt}(\tilde{c})=61$, $[25,6,11]_2$: $(1,\!2,\!2,\!3,\!6,\!11)\to3$;
    \item $\operatorname{wt}(\tilde{c})=58$, $[28,6,12]_2$: $(1,\!2,\!3,\!4,\!6,\!12)\to4182$, $(1,2,4,3,6,12)\to1588$, $(1,\!2,\!2,\!4,\!7,\!12)\to1$;
    \item $\operatorname{wt}(\tilde{c})=57$, $[29,6,13]_2$: $(1,\!1,\!2,\!4,\!8,\!13)\to1$;
    \item $\operatorname{wt}(\tilde{c})=49$, $[37,6,17]_2$: $(1,\!2,\!3,\!5,\!9,\!17)\to2$;
    \item $\operatorname{wt}(\tilde{c})=45$, $[41,6,19]_2$: $(1,\!2,\!4,\!5,\!10,\!19)\to10$, $(1,\!2,\!3,\!6,\!10,\!19)\to7$.
  \end{itemize}
  With this, we have computationally excluded the cases $\operatorname{wt}(\tilde{c})\in\{58,61,63,64\}$ for a residual code of a codeword of weight $82$.
  We have enumerated all $15$ projective $[86,7,\{41,\dots,54,57,62\}]_2$ codes. 
  It turns out that weights $53$, $54$, $57$ do not occur and that weight $62$ always occurs. Moreover, the maximum frequencies for the larger weights are 
  given by $62^1 52^1 51^1 50^2 49^4 48^3 47^6$. With these additional restrictions at hand, we computationally exclude the case $\operatorname{wt}(\tilde{c})=82$, 
  so that $C$ cannot contain a codeword $c$ of weight $82$ -- contradiction. The overall computation took 10~hours.
\end{Proof}

\begin{table}[htp]
  \begin{center}
    \begin{tabular}{rrrrrrrrrr}
      \hline
      n  & 146 & 155 & 160 & 163 & 164 &  166 & 167 & 168 \\  
      k  &   2 &   3 &   4 &   5 &   5 &    6 &   7 &   8 \\
      \# &   3 &   2 &   2 &   2 &  43 &   74 &   0 &   0 \\ 
      \hline
    \end{tabular}     
  \end{center}  
  \caption{Number of $[n,k,\{82,84,\dots,112\}]_2$ codes which are recursively obtainable from a $[146,2,W]_2$ code containing a codeword of weight $82$.}
  \label{table_168_8_82_82_64}
\end{table}  

From Lemma~\ref{lemma_binary_code}, Lemma~\ref{lemma_no_168_8_82}, and the existence of an $[55,4,40]_4$ code\footnote{Adding an arbitrary line or a double-point to a 
$(54,8,14)$ system also yields $(55,8,15)$ systems.} we directly conclude:
\begin{Theorem}$n(15)=55$.\end{Theorem}

\begin{Remark}
  In the computational approach of the proof of Lemma~\ref{lemma_no_168_8_82} we have a lot of choices. We may directly enumerate the $257$ projective $[86,7,41]_2$ 
  codes to conclude the forbidden weights for the residual code without using the results from \cite{bouyukliev2001optimal}. While the additional assumption of 
  a projective code does not give more forbidden weights in our example, we might use the maximum weight frequencies $64^1 63^1 62^1 61^1 58^1 57^2 54^2 53^2 52^2
  51^2 50^2 49^4 48^{10}$ directly from the start. We mention that one might also deduce more refined conditions like that the number of residual codewords of weight 
  at least $58$ is at most $1$ or to use a lexicographic comparison with all possible weight enumerators.
  
  There is some similarity to the linear programming method based on the split weight enumerator as used in \cite{jaffe1997brief}. We also remove certain partitions 
  from the search tree. However, we do not use linear programming but exhaustive enumeration and information on residual codes. Instead of forcing the existence 
  of codewords of a certain partition by linear programming we use available information of the possible minimum distances of (residual) codes.    
  
  The use of minimum distance bounds for residual codes in the generation tree comes with some subtleties. If we aim at a $[168,8,\{82,84,\dots,112\}]_2$ 
  code starting from the $[82,1,82]_2$ code without using weight information for the residual codes, then the partial counting information is given in Table~\ref{table_ex_recursion}. 
  However, the number of $[145,3,\{82,84,\dots,112\}]_2$, $[156,4,\{82,84,\dots,112\}]_2$, and $[162,2,\{82,84,\dots,112\}]_2$ codes is given by $4$, $30$, and $1524$, respectively. 
  The missing $3$-dimensional code is geometrically obtained by removing a double-point from a plane with multiplicity $21$.\footnote{All four non-isomorphic 
  even $[145,3,82]_2$ codes can geometrically be obtained by removing three lines from a $22$-fold plane.} Note that this code does not contain 
  a $[123,2,82]_2$ code as a subcode, but only a $[124,2,82]_2$ code. 
\end{Remark}

\begin{table}[htp]
  \begin{center}
    \begin{tabular}{rrrrrrrrrrr}
      \hline
      n  & 82 & 123 & 144 & 145 & 155 &  156 & 161 &  162 & \dots & 168 \\  
      k  &  1 &   2 &   3 &   3 &   4 &    4 &   5 &    5 & \dots &   8 \\
      \# &  1 &   1 &   1 &   3 &   3 &   28 &  23 & 1509 & \dots &   0 \\ 
      \hline
    \end{tabular}     
  \end{center}  
  \caption{Number of $[n,k,\{82,84,\dots,112\}]_2$ codes which are recursively obtainable from a $[82,1,82]_2$ code.}
  \label{table_ex_recursion}
\end{table}  

\begin{Lemma}
  \label{lemma_no_327_8_162}
  No $[327,8,162]_2$ code with maximum weight at most $218$ exists.
\end{Lemma}
\begin{Proof}
  Assume that $C$ is a $[327,8,162]_2$ code with maximum weight at most $218$. From Lemma~\ref{lemma_griesmer_divisibility} we conclude that $C$ is $2$-divisible. From 
  Lemma~\ref{lemma_residual_code} and the non-existence of the corresponding residual codes we conclude that $C$ does not have a codeword with weight $194$.
  So, consider the set $W=\{162,164,\dots,218\}\backslash\{194\}$ of possible non-zero weights.
  %% For weight 162 we observe that an [165,7,81]_2 code is an optimal Griesmer code. 2[7]-[6]-[4]-[3]-[2]-[1]    166_7_82.2: #=53 {82,...,96,114,...,124,128}
  %% For weight 178 we observe that an [149,7,73]_2 code is an optimal Griesmer code. 2[7]-[6]-[5]-[3]-[2]-[1] 
  %% For weight 186 we observe that an [141,7,69]_2 code is an optimal Griesmer code. 2[7]-[6]-[5]-[4]-[2]-[1] 
  %% For weight 190 we observe that an [137,7,67]_2 code is an optimal Griesmer code. 2[7]-[6]-[5]-[4]-[3]-[1] 
  %% For weight 192 we observe that an [135,7,66]_2 code is an optimal Griesmer code. 2[7]-[6]-[5]-[4]-[3]-[2]
  %% Since no [133,7,65]_2 code exists, there is no codeword of weight 194.
  For a codeword of weight $162$ the residual code $C'$ is a $[165,7,81]_2$ Griesmer code. We have enumerated all $53$ even $[166,7,82]_2$ codes.\footnote{Using the 
  projective dual transform with $\alpha=\tfrac{1}{2}$, $\beta=-41$, and $m=7$ these codes also correspond to $[105,7,\{32,48,64\}]_2$ codes, see \cite{bouyukliev2009classification}.} Since their 
  non-zero weights are all contained in $\{82,\dots,96,114,\dots,124,128\}$ and the maximum weight is at least $114$, the non-zero weights of $C'$ are 
  all contained in $\{81,\dots,96,113,\dots,124,127,128\}$ and the maximum weight is at least $113$.  
  %% Residual even weight code
  %% [166,7,82]
  %% [1,82]:1   [84,6,41]
  %% [2,123]:1  [43,5,21]
  %% [3,144]:1  [22,4,11]
  %% [4,155]:3  [11,3,6]
  %% [5,161]:23 [5,2,3]
  %% [6,164]:43 [2,1,2]
  %% [6,166]:53
  Starting from the unique $[162,1,W]_2$ code we constructed all 
  $[n,2,W]_2$ codes with $n\in\{275,\dots,290\}\backslash\{287,288\}$. There are no $[n',3,W]_2$ code which are obtainable from these codes and satisfy the mentioned 
  constraints for the residual codes of codewords of weight $162$. So, $C$ cannot contain a codeword of weight $162$ -- contradiction. The overall computation took 
  less than two hours.
  %% 302 | <=276 aber 278
  %% 303 | <=278 aber 279
  %% 304 | <=280 aber 282
  %% 305 | <=282 aber 285
  %% 306 | <=284 aber 289
  %% 307 | <=286 aber 289
  %% 2 Minuten Rechenzeit
\end{Proof}

From Lemma~\ref{lemma_binary_code}, Lemma~\ref{lemma_no_327_8_162}, and the existence of an $[108,4,80]_4$ code\footnote{Adding an arbitrary line or a double-point to a 
$(107,8,27)$ system also yields $(108,8,28)$ systems.} we directly conclude:
\begin{Theorem}$n_4(28)=108$.\end{Theorem}

\begin{Lemma}
  \label{lemma_no_342_8_170}
  No $[342,8,170]_2$ code with maximum weight at most $228$ exists.
\end{Lemma}
\begin{Proof}
  Assume that $C$ is a $[342,8,170]_2$ code with maximum weight at most $228$. From Lemma~\ref{lemma_griesmer_divisibility} we conclude that $C$ is $2$-divisible. 
  From Lemma~\ref{lemma_residual_code} and the non-existence of the corresponding residual codes we conclude that $C$ does not have a codeword with a weight in 
  $\{178,194,202,206,208,210\}$. So, consider the set $W=\{170,172,\dots,228\}\backslash\{178,194,202,206,208,210\}$ of possible non-zero weights.
  For a codeword of weight $170$ the residual code $C'$ is a $[172,7,85]_2$ code. We have enumerated all five even $[173,7,86]_2$ codes. Since their 
  non-zero weights are all contained in $\{86,88,94,96,118,120,126,128\}$ and the maximum weight is at least $118$, the non-zero weights of $C'$ are 
  all contained in $\{85,\dots,88,93,\dots,96,117,\dots,120,125,\dots,128\}$ and the maximum weight is at least $117$. 
  %% Residual even weight code
  %% [173,7,86] #=5 {86,88,94,96, 118,120,126,128} exactly one big 2[7]-[6]-[4]-[2] alle Solomon-Stiffler
  %% [1,86]:1 [87,6,43]  
  %% [2,129]:1 [44,5,22]
  %% [3,151]:1 [22,4,11]
  %% [4,162]:1 [11,3,6]
  %% [5,168]:2  [5,2,3]
  %% [6,171]:3  [2,1,2]
  %% [7,173]:5
  %% The complete computation took 2:04 minutes.
  Starting from the unique $[170,1,W]_2$ code we constructed all 
  $[n,2,W]_2$ codes with $n\in\{287,\dots,298\}\backslash\{291,\dots,294\}$. There are no $[n',3,W]_2$ code which are obtainable from these codes and satisfy the mentioned 
  constraints for the residual codes of codewords of weight $170$. So, $C$ cannot contain a codeword of weight $170$ -- contradiction. The overall computation took 
  less than five minutes.
  %% [1,170]:
  %% [2,287]: [55,6,27] 
  %% [2,288]: [54,6,26]
  %% [2,289]: [53,6,26]
  %% [2,290]: [52,6,25]
  %% [2,295]: [47,6,23]
  %% [2,296]: [46,6,22]
  %% [2,297]: [45,6,22]
  %% [2,298]: [44,6,21] 
  %% No extension possible
\end{Proof}

From Lemma~\ref{lemma_binary_code}, Lemma~\ref{lemma_no_342_8_170}, and the existence of an $[113,4,84]_4$ code we directly conclude:
\begin{Theorem}$n_4(29)=113$.\end{Theorem}

\section*{Acknowledgments}
The author would like to thank the anonymous reviewers for their helpful comments and feedback, which improved the presentation of the paper. Moreover I am very grateful to Simeon Ball for feedback on earlier
drafts and generously sharing his insights on additive codes. 

\section*{Epilogue}
My work on additive codes was initiated by completely solving the case of dimension $k=3.5$ and finding four infinite series where the parameters of additive codes outperform those of
linear codes for an integer dimension $k$, see \url{https://arxiv.org/pdf/2410.07650v1}. Slighly before the paper was rejected I managed to generalize and formalize the underlying
techniques to the case of general additive codes, see \url{https://arxiv.org/abs/2412.14615v1}. Obviously, the referees were right that the initial 4-pager was written rather informal
and that it should be possible to obtain more results. The present paper is a compromise between this short and the extensive version with more than 100 pages, by focusing just on binary
codes and the fact that for large distances a Griesmer-type bound is the right answer to the question on the minimum possible length. After a resubmission and a revision it was rejected
due to the lack of a closed-form or asymptotic expression for the minimum possible distance, building too closely on existing methods, and not providing enough new insights,
cf.~\url{https://arxiv.org/pdf/2410.07650v3}. In the meantime I could build up upon the results of \cite{kurz2024additive} and show that a Griesmer-type bound can also always be attained
for sufficiently large minimum distances for linear codes in the so-called $b$-symbol metric, see \url{https://arxiv.org/abs/2507.07728}.
I would like to apologize to the editors and referees of IEEE Transactions on Information Theory for my produced chaos and thank them for their work.

\end{document}